\documentclass[12pt]{article}%
\usepackage{graphicx}
\usepackage{amsmath}
\usepackage{amsfonts}
\usepackage{amssymb}%
\providecommand{\U}[1]{\protect\rule{.1in}{.1in}}
\newtheorem{theorem}{Theorem}

\newtheorem{proposition}[theorem]{Proposition}

\newenvironment{proof}[1][Proof]{\noindent\textbf{#1.} }{\ \rule{0.5em}{0.5em}}
\begin{document}

\title{\textbf{Queuing Networks with Varying Topology -- A Mean-Field Approach.}}
\author{Fran\c{c}ois Baccelli$^{\ddag}$, Alexandre Rybko$^{\dag\dag}$ and Senya
Shlosman$^{\dag,\dag\dag}$\\$^{\ddag}$Dept. of Math., UT Austin, Austin, USA\\$^{\dag}$Aix Marseille Universit\'{e}, CNRS, CPT, UMR 7332, \\13288 Marseille, France,\\$^{\dag}$Universit\'{e} de Toulon, CNRS, CPT, UMR 7332, \\83957 La Garde, France, \\$^{\dag\dag}$Inst. of the Information Transmission Problems,\\RAS, Moscow, Russia }
\maketitle

\begin{abstract}
We consider the queuing networks, which are made from servers, exchanging
their positions. The customers, using the network, try to reach their
destinations, which is complicated by the movements of the servers, taking
their customers with them, while they wait for the service. We develop the
general theory of such networks, and we establish the convergence of the
symmetrized version of the network to the Non-Linear Markov Process.

\end{abstract}

\section{Introduction}

In this paper we consider a model of queuing network containing servers moving
on the set of nodes of some graph $G$, in such a way that at any time, a node
harbors a single server. Customers enter the network at each entrance node.
When it arrives to some node, a customer joins the queue currently harbored by
this node. Customer $c$ also has a designated exit node $D\left(  c\right)  $,
which he needs to reach in order to exit the network. In order to reach its
destination, a customer visits a series of intermediate servers: when at
server $v$, a customer waits in the associated queue before being served. The
waiting time depends on the service discipline at $v$. Once a customer $c$
being served by $v$ leaves, he is sent to the server $v^{\prime},$ located at
the adjacent node, which is the closest to the destination node $D\left(
c\right)  $. Once it gets to $D\left(  c\right)  $, the customer leaves the system.

The main feature of the network we are considering is that servers are moving
on the graph. So while customers are waiting to be served, two adjacent
servers can move by simultaneously swapping their locations. When two servers
swap, each one carries its queue with it. So if the server $v$, currently
containing $c,$ moves, then the distance between $c$ and its destination
$D\left(  c\right)  $ might change, in spite of the fact that customer $c$ has
not yet completed its service at $v$.

In such networks with moving nodes new effects take place, which are not
encountered in the usual situations with stationary nodes. For example, it can
happen that `very nice' networks -- i.e. networks with fast servers and low
load -- become unstable once the nodes start to move. Here the instability
means that the queues becomes longer and longer with time, for a large
network. In contrast, for the same parameters, the queues remain finite in the
network with stationary servers.

This instability, which appear as a result of the movement of the servers,
will be a subject of our forthcoming papers, \cite{BRS}. In the present paper
we develop the general qualitative theory of such networks, and we focus on
the mean-field approach to them. The main result can be described as follows.
We start with the definition of the class of networks with jumping nodes. The
network can be finite or infinite. In order to be able to treat the problem,
we consider the mean-field version of it, which consists of $N$ copies of the
network, interconnected in a mean-field manner. We show that as $N$ increases,
the limiting object becomes a Non-Linear Markov Process (NLMP) in time. The
ergodic properties of this process determine the stability or instability of
our network. They will be investigated in our forthcoming papers, \cite{BRS}.
Here we will establish the existence of the NLMP and the convergence to it.

We now recall what is meant by \textbf{Non-Linear Markov Processes}. We do
this for the simplest case of discrete time Markov chains, taking values in a
finite set $S,$ $\left\vert S\right\vert =k.$ For the general case see
\cite{RS}, Sect. 2 and 3. In the discrete case the set of states of our Markov
chain is the simplex $\Delta_{k}$ of all probability measures on $S,$
$\Delta_{k}=\left\{  \mu=\left(  p_{1},...,p_{k}\right)  :p_{i}\geq
0,p_{1}+...+p_{k}=1\right\}  ,$ while the Markov evolution defines a map
$P:\Delta_{k}\rightarrow\Delta_{k}.$ In the case of usual Markov chains the
map $P$ is linear with $P$ coinciding with the matrix of transition
probabilities. A non-linear Markov chain is defined by a \textit{family} of
transition probability matrices $P_{\mu},$ $\mu\in\Delta_{k},$ so that the
matrix entry $P_{\mu}\left(  i,j\right)  $ is the probability of going from
$i$ to $j$ in one step, starting \textit{in the state}\textbf{\ }$\mu.$ The
\textit{(non-linear)} map $P$ is then defined by $P\left(  \mu\right)  =\mu
P_{\mu}.$

The ergodic properties of linear Markov chains are settled by the
Perron-Frobenius theorem. In particular, if the linear map $P$ is such that
the image $P\left(  \Delta_{k}\right)  $ belongs to the interior $Int\,\left(
\Delta_{k}\right)  $ of $\Delta_{k},$ then there is precisely one point
$\mu\in Int\,\left(  \Delta_{k}\right)  ,$ such that $P\left(  \mu\right)
=\mu,$ and for every $\nu\in\,\Delta_{k}$ we have the convergence
$P^{n}\left(  \nu\right)  \rightarrow\mu$ as $n\rightarrow\infty.$ In case $P$
is non-linear, we are dealing with a more or less arbitrary dynamical system
on $\Delta_{k}$, and the question about the stationary states of the chain or
about measures on $\Delta_{k}$ invariant under $P$ cannot be settled in general.

\section{Formulation of the Main Result}

\subsection{\label{descr} Network Description and the Main Result in a
Preliminary Form}

We consider a queueing network with jumping nodes on a connected graph
\[
G=\left[  V\left(  G\right)  ,E\left(  G\right)  \right]  .
\]
This means that at every node $v\in V$, at any time, there is one server with
a queue $q_{v}$ of customers waiting there for service. Every customer $c$ at
$v$ carries its destination,\textbf{ }$D\left(  c\right)  \in V\left(
G\right)  .$ The goal of the customer is to reach its destination. In order to
get there, a customer completing its service at $v$ jumps along the edges of
$G$ to one of the nodes of $G$ which is closest (in the graph distance) to
$D\left(  c\right)  $ where he joins the queue of the server currently
harbored by this node. Once the destination of a customer is reached, he
leaves the network. In the meantime, the servers of our network can jump. More
precisely, the two servers at $v,v^{\prime},$ which are neighbors in $G$ can
exchange their positions with rate $\beta_{vv^{\prime}}$. The queues $q_{v}$
and $q_{v^{\prime}}$ then exchange their positions as well. Of course, such an
exchange may bring some customers of $q_{v}$ and $q_{v^{\prime}}$ closer to
their destinations, and some other customers further away from their
destinations. We suppose that the rates $\beta_{v^{\prime}v}$ of such jumps
are uniformly bounded by a constant:
\begin{equation}
\left\vert \beta_{v^{\prime}v}\right\vert <\mathbf{\beta}. \label{30}%
\end{equation}
Our goal is to study the behavior of such a network.

The graphs $G$ we are interested in can be finite or infinite. The case of
finite graphs is easier, while the infinite case requires certain extra
technical points. In particular, when we talk about the functions of the
states of our network in the infinite graph case, we will always assume that
they are either local or quasi-local, with dependence on far-away nodes
decaying exponentially fast with the distance. This exponential decay property
will be conserved by our dynamics.

We will suppose that the degrees of the vertices of $G$ are finite and
uniformly bounded by some constant $D\left(  G\right)  $. Of course, this
automatically holds in the finite graph case.

In order to make our network tractable, we will study its \textit{symmetrized}%
, or mean-field, modification. This means that we pass from the graph $G$ to
its mean-field version, the graph $G_{N}=G\times\left\{  1,...,N\right\}  ,$
and eventually take the limit $N\rightarrow\infty.$ By definition, the graph
$G_{N}$ has the set of vertices $V\left(  G_{N}\right)  =V\left(  G\right)
\times\left\{  1,...,N\right\}  ;$ two vertices $\left(  v,k\right)  ,\left(
v^{\prime},k^{\prime}\right)  \in V\left(  G_{N}\right)  $ define an edge in
$E\left(  G_{N}\right)  $ iff $\left(  v,v^{\prime}\right)  \in E\left(
G\right)  .$ As we shall see, the restriction of our state process from
$G\times\left\{  1,...,N\right\}  $ to the subgraph $G\equiv G\times\left\{
1\right\}  $ goes, as $N\rightarrow\infty,$ to a Non-Linear Markov Process on
$G,$ which is a central object of our study. The limiting network
$\mathcal{K}$ (which coincides, in a sense, with the above mentioned NLMP)
will be the limit of the networks $\mathcal{K}_{N}$ on $G\times\left\{
1,...,N\right\}  $.

In our model the servers can exchange their positions (bringing all the
customers queuing at them to the new locations). The above rate of exchange
$\beta_{vv^{\prime}}$ should be renormalized, as we pass from $G$ to $G_{N},$
in order for the limit to exist. So for a server located at $\left(
v,k\right)  \in V\left(  G\times\left\{  1,...,N\right\}  \right)  ,$ the rate
of transposition with the server $\left(  v^{\prime},k^{\prime}\right)  $,
where the node $v^{\prime}$ is a neighbor of the node $v,$ is given by
\[
\frac{\beta_{vv^{\prime}}}{N}.
\]
This implies that the server at $\left(  v,k\right)  $ will jump to the set of
nodes $v^{\prime}\times\left\{  1,...,N\right\}  $ with the rate
$\beta_{vv^{\prime}},$ independent of $N.$

Every server has a queue of customers. Each customer has a class,
$\varkappa\in K.$ These classes are letters of some finite alphabet $K$. If a
customer of class $\varkappa$ completes his service on the server at $v$ and
goes to server $u$, it then gets a new class $\varkappa^{\prime}%
=\mathcal{T}\left(  \varkappa;v,u\right)  .$ Once a server finishes serving a
customer, it chooses another one from the queue, according to the class of the
customers present in the queue and the service discipline. It can happen that
the service of a customer is interrupted if a customer with higher priority
comes, and then the interrupted service is resumed after the appropriate time.

The service time distribution $\eta$ depends on the class of the customer and
on the server $v$, $\eta=\eta\left(  \varkappa,v\right)  $. We do not suppose
that $\eta$ is exponential.

Every customer $c$ in our network $\mathcal{K}_{N}$ has its destination node,
$D\left(  c\right)  =w\in V\left(  G\right)  .$ In spite of the fact that our
servers do change their positions, this location $D\left(  c\right)  $ does
not change with time. The customer $c$ tries to get to its destination node.
In order to do so, if it is located at $\left(  v,k\right)  $ and finishes its
service there, then it goes to the server at $\left(  v^{\prime},n\right)  $,
where $v^{\prime}\in G$ is the neighbor of $v$ which is the closest to
$D\left(  c\right)  $. If there are several such $v^{\prime},$ one is chosen
uniformly. The coordinate $n\in\left\{  1,...,N\right\}  $ is chosen uniformly
as well. If at the moment of the end of the service it so happens that $v$ is
at distance $1$ from $D\left(  c\right)  $ or that $v$ coincides with
$D\left(  c\right)  ,$ then the customer leaves the network. However, if the
customer $c$ is waiting for service at the server $\left(  v,k\right)  ,$ then
nothing happens with him even if $v=D\left(  c\right)  ;$ it can be that at a
later moment this server will drift away from the node $D\left(  c\right)  ,$
and the distance between $c$ and its destination $D\left(  c\right)  =w$ will
increase during this waiting time.

A formal definition of the Markov process describing the evolution of the
network $\mathcal{K}_{N}$ will be provided later, in Section \ref{finite}. Our
main result is the proof of the convergence of the network $\mathcal{K}_{N}$
to the Non-Linear Markov Process -- which is the limiting mean-field system.
The formulation of our main theorem is given in Section \ref{main stat}.

The rest of the paper is organized as follows. In Subsection \ref{comb} we
present the description of the state space of our NLMP. In Subsection
\ref{main stat} we then describe its possible jumps, write down its evolution
equation and finally formulate our main result about the existence of the NLMP
and the convergence of the networks $\mathcal{K}_{N}$ to it, as $N\rightarrow
\infty.$ In Section 3 we prove the existence theorem for the NLMP. Next
Section 4 is devoted to various compactification arguments. We use these
arguments in Section 5 to check the applicability of the Trotter-Kurtz
theorem, thus proving the convergence part of our main result. 

\subsection{\label{comb} The State Space of the Mean-Field Limit}

We describe below the state space of the mean-field limit, which will be
referred to as the \emph{Comb}.

At any given time, at each node $v\in G$, we have a finite \textit{ordered}
queue $q_{v}$ of customers, $q_{v}=\left\{  c_{i}\right\}  \equiv\left\{
c_{i}^{v}\right\}  \equiv\left\{  c_{1}^{v},...,c_{l\left(  q_{v}\right)
}^{v}\right\}  ,$ where $l\left(  q_{v}\right)  $ is the length of the queue
$q_{v}$. The customers are ordered according to their arrival times to $v.$
The information each customer carries consists of

\begin{enumerate}
\item Its class $\varkappa_{i}\equiv\varkappa_{i}\left(  c_{i}\right)  \in K,$
($\left\vert K\right\vert <\infty$) and

\item The final address $v_{i}=D\left(  c_{i}\right)  \in V(G)$ which the
customer wants to reach. The class $\varkappa$ of the customer can change as a
result of the service process. We will denote by $\bar{\varkappa}\left(
c\right)  $ the class of the customer $c$ once its service at the current
server is over. In what follows we consider only conservative disciplines.
This means that the server cannot be idle if the queue is not empty.

\item We will denote by $C\left(  q_{v}\right)  $ the customer of queue
$q_{v}$ which is being served, and we denote by $\tau\left(  C\left(
q_{v}\right)  \right)  $ the amount of time this customer already spent being
served. We need to keep track of it since our service times are not
exponential in general. It can happen that in the queue $q_{v}$ there are
customers of lower priority than $C\left(  q_{v}\right)  ,$ which already
received some service, but whose service is postponed due to the arrival of
higher priority customers.

\item Let $i^{\ast}\left(  q_{v}\right)  $ be the location of the customer
$C\left(  q_{v}\right)  $ in the queue $q_{v},$ i.e. $C\left(  q_{v}\right)
\equiv c_{i^{\ast}\left(  q_{v}\right)  }^{v}.$ The service discipline is the
rule $R_{v}$ to choose the location $i^{\ast}\left(  q_{v}\right)  $ of the
customer which has to be served. In what follows we will suppose that the rule
$R_{v}$ is some function of the sequence of classes of our customers,
$\varkappa_{1},...,\varkappa_{l\left(  q_{v}\right)  }$ and of the sequence of
their destinations, $D\left(  c_{1}^{v}\right)  ,...,D\left(  c_{l\left(
q_{v}\right)  }^{v}\right)  ,$ so that
\[
i^{\ast}\left(  q_{v}\right)  =R_{v}\left[  \left\{  \varkappa_{1}%
,...,\varkappa_{l\left(  q_{v}\right)  }\right\}  ,\left\{  D\left(  c_{1}%
^{v}\right)  ,...,D\left(  c_{l\left(  q_{v}\right)  }^{v}\right)  \right\}
\right]  .
\]
We assume that the function $R_{v}$ depends on $\left\{  D\left(  c_{1}%
^{v}\right)  ,...,D\left(  c_{l\left(  q_{v}\right)  }^{v}\right)  \right\}  $
only through the relative distances $dist\left(  D(c_{i}^{v}),v\right)  $.

\item The amount of service already acquired by the customers $c_{1}%
^{v},...,c_{l\left(  q_{v}\right)  }^{v}$ will be denoted by $\tau
_{1},...,\tau_{l\left(  q_{v}\right)  }.$ At the given time, the only $\tau$
variable which is growing is $\tau_{i^{\ast}\left(  q_{v}\right)  }\equiv
\tau\left[  C\left(  q_{v}\right)  \right]  \equiv\tau\left[  c_{i^{\ast
}\left(  q_{v}\right)  }^{v}\right]  $. Sometime we will write $c^{v}\equiv
c^{v}\left(  \varkappa,\tau,v^{\prime}\right)  $ for a customer at $v$ of
class $\varkappa,$ who has already received $\tau$ units of service and whose
destination is $v^{\prime}.$
\end{enumerate}

\noindent The space of possible queue states at $v$ is denoted by $M\left(
v\right)  $. The `coordinates' in $M\left(  v\right)  $ are listed in the five
items above. So $M\left(  v\right)  $ is a countable union of
finite-dimensional positive orthants; it will be referred to as the Comb.

Let $M=\prod_{v\in G}M\left(  v\right)  .$ The measure $\mu$ of the NLMP is
defined on this product space $M,$ which is hence a product of Combs. It turns
out that we will encounter only the product measures on $M$. We will discuss
this point below; see also \cite{PRS}, where we prove a simple extension of
the de Finetti's theorem.

\subsection{Possible State Jumps of the Mean-Field Limit and their Rates}

We give here a brief summary of the jumps of the NLMP on the Comb and their rates.

\subsubsection{The Arrival of a New (external) Customer}

A customer $c^{vv^{\prime}}$ of class $\varkappa\in K$ arrives at node $v,$
with destination $D\left(  v\right)  =v^{\prime}$ with rate $\lambda
=\lambda\left(  \varkappa,v,v^{\prime}\right)  $. We suppose that
\[
\sum_{\varkappa,v^{\prime}}\lambda\left(  \varkappa,v,v^{\prime}\right)  <C,
\]
uniformly in $v.$ As far as the notation is concerned, we will say that the
queue state $q=\left\{  q_{u},u\in G\right\}  $ changes to $q^{\prime
}=\left\{  q_{u},u\in G\right\}  \oplus c^{vv^{\prime}}.$ The associated jump
rate $\sigma_{e}\left(  q,q^{\prime}\right)  $ is
\begin{equation}
\sigma_{e}\left(  q,q^{\prime}\right)  =\lambda\left(  \varkappa,v,v^{\prime
}\right)  . \label{42}%
\end{equation}

\subsubsection{\label{serv} Service Completion}

It is easy to see that the customer in service at node $v,$ who received the
amount $\tau$ of service time, finishes his service at $v$ with the rate
$\frac{\mathcal{F}_{\varkappa\left(  C\left(  q_{v}\right)  \right)
,v}^{\prime}\left(  \tau\right)  }{1-\mathcal{F}_{\varkappa\left(  C\left(
q_{v}\right)  \right)  ,v}\left(  \tau\right)  },$ where $F_{\varkappa,v}$
denotes the distribution function of the service time. For future use we
suppose that this rate has a limit as $\tau\rightarrow\infty.$ We also suppose
that it is uniformly bounded by a constant $F<\infty.$ The queue state
$q=\left\{  q_{u},u\in G\right\}  $ changes to $q^{\prime}=
\left\{  q_{u},u\in G\right\}  \ominus C\left(  q_{v}\right)  ,$ so we denote
this rate\footnote{Queuing theorists might be surprised by these departures
without arrivals whereas customers do not necessarily leave the network. As we
shall see, in the mean-field limit, any single departure from $v$ to
$v^{\prime}$ has no effect on the state of the queues of $v^{\prime}$ because
of the uniform routing to the mean-field copies. However, the sum of the
departure processes from all copies of the servers at $v$ leads to a positive
arrival rate from $v$ to $v^{\prime}$ which is evaluated in subsection
\ref{ssarc} below.} by
\begin{equation}
\sigma_{f}\left(  q,q^{\prime}\right)  =\frac{\mathcal{F}_{\varkappa\left(
C\left(  q_{v}\right)  \right)  ,v}^{\prime}\left(  \tau\right)
}{1-\mathcal{F}_{\varkappa\left(  C\left(  q_{v}\right)  \right)  ,v}\left(
\tau\right)  }\leq\mathcal{F}. \label{41}%
\end{equation}

\subsubsection{Servers Jumping}

Let the server at $v$ jump and exchange with the one at $v^{\prime}.$ As a
result, the queue $q_{v}$ is replaced by a (random) queue $Q$ distributed
according to the distribution law $\mu_{v^{\prime}}\left(  dQ\right)  .$ So
the state changes from $q=\left\{  q_{u},u\in G\right\}  $ to $q^{\prime
}=\left\{  Q,q_{u},u\neq v\in G\right\}  $ with rate
\begin{equation}
\sigma_{ex}\left(  q,q^{\prime}\right)  =\beta_{vv^{\prime}}\mu_{v^{\prime}%
}\left(  dQ\right)  . \label{21}%
\end{equation}

\subsubsection{The Arrival of the Transit Customers}

\label{ssarc}

Suppose we are at node $v^{\prime},$ and that a customer $c^{v}$ of class
$\varkappa$ located in the server of the neighboring node $v$ completes his
service. What are the chances that this customer joins node $v^{\prime}%
\in\mathcal{N}\left(  v\right)  $ to be served there? Here we denote by
$\mathcal{N}\left(  v\right)  $ the set of all vertices of the graph $G,$
which are neighbors of $v$. For this to happen it is necessary that
\[
\mathrm{dist}(v,D\left(  c^{v}\right)  )=\mathrm{dist}(v^{\prime},D\left(
c^{v}\right)  )+1,\text{ and }\mathrm{dist}(v^{\prime},D\left(  c^{v}\right)
)>0.
\]
If there are several such nodes in $N\left(  v\right)  $, then all of them
have the same chance. In the case $dist(v,D\left(  c^{v}\right)  )=1$, the
customer $c^{v}$ goes to the node $D\left(  c^{v}\right)  $ and leaves the
network immediately. Let $E\left(  v,D\left(  c^{v}\right)  \right)  $ be the
number of such nodes:
\[
E\left(  v,D\left(  c^{v}\right)  \right)  =\#\left\{  w\in\mathcal{N}\left(
v\right)  :\mathrm{dist}(v,D\left(  c^{v}\right)  )=\mathrm{dist}(w,D\left(
c^{v}\right)  )+1\right\}  .
\]
Thus, for every pair $v,D\in G$ of sites with $\mathrm{dist}(v,D)>1$, we
define the function $e_{v,D}$ on the sites $w\in G:$%
\begin{equation}
e_{v,D}\left(  w\right)  =\left\{
\begin{array}
[c]{cc}%
\frac{1}{E\left(  v,D\right)  } & \text{if }w\in\mathcal{N}\left(  v\right)
:\mathrm{dist}(v,D)=\mathrm{dist}(w,D)+1\\
0 & \text{otherwise}%
\end{array}
\right.  . \label{11}%
\end{equation}
Then, in the state $\mu$, the rate of the transit customers of class
$\varkappa$ arriving to node $v^{\prime}$ is given by
\begin{align}
&  \sigma_{tr}\left(  q,q\oplus c^{v^{\prime},\varkappa}\right)  \equiv
\sigma_{tr}^{\mu}\left(  q,q\oplus c^{v^{\prime},\varkappa}\right) \nonumber\\
&  =\sum_{v\in\mathcal{N}\left(  v^{\prime}\right)  }\int d\mu\left(
q_{v}\right)  e_{v,D\left(  C\left(  q_{v}\right)  \right)  }\left(
v^{\prime}\right)  \frac{\mathcal{F}_{\varkappa,v}^{\prime}\left(  \tau\left(
C\left(  q_{v}\right)  \right)  \right)  }{1-\mathcal{F}_{\varkappa,v}\left(
\tau\left(  C\left(  q_{v}\right)  \right)  \right)  }\delta\left(
\bar{\varkappa}\left(  C\left(  q_{v}\right)  \right)  ,\varkappa\right)  .
\label{0012}%
\end{align}
Here $\bar{\varkappa}$ is the class the customer $C\left(  q_{v}\right)  $
gets after his service is completed at $v$. Such a customer $c^{v^{\prime
},\varkappa}$ just arrived to $v^{\prime}$, if generated by the customer
$C\left(  q_{v}\right)  ,$ has the site $D\left(  C\left(  q_{v}\right)
\right)  $ as its destination.

Note that the two rates $\left(  \ref{21}\right)  ,$ $\left(  \ref{0012}%
\right)  $ do depend on the measure $\mu,$ which is the source of the
non-linearity of our process.

\subsection{Evolution Equations of the Non-Linear Markov Process}

We start with a more precise description of the state space of the initial network.

Denote by $v\in V\left(  G\right)  $ a vertex of $G$. The state of the server
$q$ at $v$ is

\begin{itemize}
\item $w$, the sequence of customers present in its queue, ordered according
to the times of their arrivals (we recall that customers belong to classes, so
that we need the sequence of customer classes to represent the state of the queue);

\item the vector describing the amount of service already obtained by these
customers; the dimension of this vector is the length $\Vert w\Vert$ of $w$.
\end{itemize}

\noindent The $i$-th coordinate of this vector will be denoted by $\tau_{i}$.
For example, for the FIFO discipline only the first coordinate can be
non-zero. For the LIFO discipline all coordinates are positive in general.

Thus, $q_{v}$ is a point in $M\left(  v\right)  $, which is the disjoint union
of all positive orthants $\mathbb{R}_{w}^{+}$, with $w$ ranging over the set
of all finite sequences of customer classes. For $w$ the empty sequence, the
corresponding orthant $\mathbb{R}_{\emptyset}^{+}$ is a single point.

Every orthant $\mathbb{R}_{w}^{+}$ with $|w|=n>0$ is equipped with a vector
field $r=\left\{  r\left(  x\right)  ,x\in\mathbb{R}_{n}^{+}\right\}  $. The
coordinate $r_{i}\left(  x\right)  $ of $r\left(  x\right)  $ represents the
fraction of the processor power spent on customer $c_{i}$, $i=1,\ldots,n$.
This fraction is a function of the current state $x$ of the queue. We have
$\sum_{i}r_{i}\left(  x\right)  \equiv1.$ In what follows we will consider
only disciplines where exactly one coordinate of the vector $r$ is $1,$ while
the rest of them are $0.$ The vectors $r\left(  x\right)  $ are defined by the
service discipline. For example, for FIFO $r_{1}\left(  x\right)  =1,$
$r_{i}\left(  x\right)  =0$ for $i>1.$ In our general notation $r_{i}\left(
x\right)  =1$ iff $i=i^{\ast}\left(  x\right)  .$

There are two natural maps between the spaces $\mathbb{R}_{w}^{+}.$ One is the
embedding
\begin{equation}
\chi:\mathbb{R}_{w}^{+}\rightarrow\mathbb{R}_{w\cup c}^{+}, \label{51}%
\end{equation}
corresponding to the arrival of the new customer $c$; it is given by
$\chi\left(  x\right)  =\left(  x,0\right)  .$ The other one is the
projection,
\begin{equation}
\psi:\mathbb{R}_{w}^{+}\rightarrow\mathbb{R}_{w\smallsetminus c_{i^{\ast
}\left(  x\right)  }}^{+}, \label{52}%
\end{equation}
corresponding to the completion of the service of the customer $c_{i^{\ast
}\left(  x\right)  },$ currently served. It is given by $\psi\left(  x\right)
=\left(  x_{1},...,x_{i^{\ast}\left(  x\right)  -1},x_{i^{\ast}\left(
x\right)  +1},...,x_{\left\vert w\right\vert }\right)  .$ For $\left\vert
w\right\vert =0$ the space $\mathbb{R}_{\varnothing}^{+}$ is a point, and the
map $\psi:\mathbb{R}_{\varnothing}^{+}\rightarrow\mathbb{R}_{\varnothing}^{+}$
is the identity.

The third natural map $\zeta_{vv^{\prime}}:M\left(  v\right)  \times M\left(
v^{\prime}\right)  \rightarrow M\left(  v\right)  \times M\left(  v^{\prime
}\right)  $ is defined for every ordered pair $v,v^{\prime}$ of neighboring
nodes. It corresponds to the jump of a customer, who has completed his service
at $v$, to $v^{\prime},$ where he is going to be served next. It is defined as
follows: if the destination $D\left(  C\left(  q_{v}\right)  \right)  $ of the
attended customer $C\left(  q_{v}\right)  $ of the queue $q_{v}$ is different
from $v^{\prime},$ then
\begin{equation}
\zeta_{vv^{\prime}}\left(  q_{v},q_{v^{\prime}}\right)  =\left(  q_{v}\ominus
C\left(  q_{v}\right)  ,q_{v^{\prime}}\oplus c\left(  C\left(  q_{v}\right)
\right)  \right)  , \label{53}%
\end{equation}
where the customer $c\left(  C\left(  q_{v}\right)  \right)  $ has the properties:

\begin{enumerate}
\item $D\left(  c\left(  C\left(  q_{v}\right)  \right)  \right)  =D\left(
C\left(  q_{v}\right)  \right)  ,$

\item $\varkappa\left(  c\left(  C\left(  q_{v}\right)  \right)  \right)
=\mathcal{T}\left(  \varkappa\left(  C\left(  q_{v}\right)  \right)
;v,v^{\prime}\right)  ,$

\item $\tau\left(  c\left(  C\left(  q_{v}\right)  \right)  \right)  =0.$
\end{enumerate}

\noindent If $D\left(  C\left(  q_{v}\right)  \right)  =v^{\prime}$ or if
$q_{v}=\varnothing,$ then we put $\zeta_{vv^{\prime}}\left(  q_{v}%
,q_{v^{\prime}}\right)  =\left(  q_{v},q_{v^{\prime}}\right)  .$

As already explained, in order to make our problem tractable, we have to pass
from the graph $G$ studied above to the mean-field graphs $G_{N},$ mentioned
earlier. They are obtained from $G$ by taking $N$ disjoint copies of $G$ and
by interconnecting them in a mean-field manner. Let $\Omega_{N}$ be the
infinitesimal operator of the corresponding continuous time Markov process. We
want to pass to the limit $N\rightarrow\infty,$ in the hope that in the limit,
the nature of the process will become simpler. The key observation is that if
this limit exists, then the arrivals to each server at every time should be a
Poisson point process (with time dependent rate function). Indeed, the flow to
every server is the sum of $N$ flows of rates $\sim\frac{1}{N}.$ Since the
probability that a customer, served at a given node revisits this node, goes
to zero as $N\rightarrow\infty,$ the arrivals to a given server in disjoint
intervals are asymptotically independent in this limit.

In order to check that the limit $N\rightarrow\infty$ exists, we will formally
write down the limiting infinitesimal generator $\Omega.$ We will then show
that it defines a (non-linear) Markov process. Finally, we will check that the
convergence $\Omega_{N}\rightarrow\Omega$ is such that the Trotter-Kurtz
theorem applies.

Let $M=\Pi_{v}M\left(  v\right)  $. We first write the evolution of the
measure $\mu$ on $M$ when $\mu$ is a product measure, $\mu=\Pi\mu_{v}$. For a
warm-up we first consider now the case when the measure $\mu$ has a density.
The general situation will be treated below, see Proposition \ref{mu}.

For $q\in M\left(  v\right)  $ denote by $e\left(  q\right)  $ the last
customer in the queue $q$, by $l\left(  q\right)  $ length of the queue. Then
the last customer $e\left(  q\right)  $ in the queue can be also denoted by
$c_{l\left(  q\right)  }$ and the quantity $\tau\left(  e\left(  q\right)
\right)  $ denotes the time this customer already was served at $v.$

We then have%

\begin{equation}
\frac{d}{dt}\mu_{v}\left(  q_{v},t\right)  =\mathcal{A}+\mathcal{B}%
+\mathcal{C}+\mathcal{D}+\mathcal{E} \label{001}%
\end{equation}
with
\begin{equation}
\mathcal{A}=-\frac{d}{dr_{i^{\ast}\left(  q_{v}\right)  }\left(  q_{v}\right)
}\mu_{v}\left(  q_{v},t\right)  \label{06}%
\end{equation}
the derivative along the direction $r\left(  q_{v}\right)  $;
\begin{equation}
\mathcal{B}=\delta\left(  0,\tau\left(  e\left(  q_{v}\right)  \right)
\right)  \mu_{v}\left(  q_{v}\ominus e\left(  q_{v}\right)  ,t\right)  \left[
\sigma_{tr}\left(  q_{v}\ominus e\left(  q_{v}\right)  ,q_{v}\right)
+\sigma_{e}\left(  q_{v}\ominus e\left(  q_{v}\right)  ,q_{v}\right)  \right]
\label{01}%
\end{equation}
where $q_{v}$ is created from $q_{v}\ominus e\left(  q_{v}\right)  $ by the
arrival of $e\left(  q_{v}\right)  $ from $v^{\prime}$, and $\delta\left(
0,\tau\left(  e\left(  q_{v}\right)  \right)  \right)  $ takes into account
the fact that if the last customer $e\left(  q_{v}\right)  $ has already
received some amount of service, then he cannot arrive from the outside (see
$\left(  \ref{0012}\right)  $ and $\left(  \ref{42}\right)  $);
\begin{equation}
\mathcal{C}=-\mu_{v}\left(  q_{v},t\right)  \sum_{q_{v}^{\prime}}\left[
\sigma_{tr}\left(  q_{v},q_{v}^{\prime}\right)  +\sigma_{e}\left(  q_{v}%
,q_{v}^{\prime}\right)  \right]  ,
\end{equation}
which corresponds to changes in queue $q_{v}$ due to customers arriving from
the outside and from other servers;
\begin{equation}
\mathcal{D}=\int_{q_{v}^{\prime}:q_{v}^{\prime}\ominus C\left(  q_{v}^{\prime
}\right)  =q_{v}}d\mu_{v}\left(  q_{v}^{\prime},t\right)  \sigma_{f}\left(
q_{v}^{\prime},q_{v}^{\prime}\ominus C\left(  q_{v}^{\prime}\right)  \right)
-\mu_{v}\left(  q_{v},t\right)  \sigma_{f}\left(  q_{v},q_{v}\ominus C\left(
q_{v}\right)  \right)  ,
\end{equation}
where the first term describes the situation where the queue state arises
$q_{v}$ after a customer was served in a queue $q_{v}^{\prime}$ (longer by one
customer), such that $q_{v}^{\prime}\ominus C\left(  q_{v}^{\prime}\right)
=q_{v},$ while the second term describes the completion of service of a
customer in $q_{v}$;
\begin{equation}
\mathcal{E}=\sum_{v^{\prime}\text{n.n.}v}\beta_{vv^{\prime}}\left[
\mu_{v^{\prime}}\left(  q_{v},t\right)  -\mu_{v}\left(  q_{v},t\right)
\right]  , \label{02}%
\end{equation}
where the $\beta$-s are the rates of exchange of the servers.

\subsection{Main Result}

\label{main stat}

Before stating the main result, we need some observations on the states of the
network. To compare the networks $\mathcal{K}_{N}$ and the limiting network
$\mathcal{K}$, it is desirable that their states are described by probability
distributions on the same space. This is in fact easily achievable, due to the
permutation symmetry of the networks $\mathcal{K}_{N}$. Indeed, if we assume
that the initial state of $\mathcal{K}_{N}$ is $\prod_{v\in G}\mathcal{S}%
_{v,N}-$invariant -- where each permutation group $\mathcal{S}_{v,N}$ permutes
the $N$ servers at the node $v$ -- then, evidently, so is the state at every
later time. After the factorization by the permutation group $\left(
\mathcal{S}_{N}\right)  ^{G}$ the configuration at any vertex $v\in G$ can be
conveniently described by an atomic probability measure $\Delta_{N}^{v}$ on
$M\left(  v\right)  $, of the form $\sum_{k=1}^{N}\frac{1}{N}\delta\left(
q_{v,k},\tau\right)  ,$ where $\tau$ is the vector of already received
services in queue $q_{v,k}$. We put $\Delta_{N}=\left\{  \Delta_{N}%
^{v}\right\}  .$

We study the limit of the networks $\mathcal{K}_{N}$ as $N\rightarrow\infty$.
For this limit to exist, we need to choose the initial states of the networks
$\mathcal{K}_{N}$ appropriately. So we suppose that the initial states
$\nu_{N}$ of our networks $\mathcal{K}_{N}$ -- which are atomic measures on
$M$ with atom weight $1/N$ -- converge to the state $\nu$ of the limiting
network $\mathcal{K}.$

We are now in a position to state the main result:

\begin{theorem}
\label{main result} Let $S_{N,t}$ be the semigroup $\exp\left\{  t\Omega
_{N}\right\}  $ defined by the generator $\Omega_{N}$ of the network
$\mathcal{K}_{N}$, described in Section \ref{descr} (the operator $\Omega_{N}$
is formally defined in $\left(  \ref{031}-\ref{034}\right)  $). Let $S_{t}$ be
the semigroup $\exp\left\{  t\Omega\right\}  ,$ defined by the generator
$\Omega$ of the network $\mathcal{K}$ (the operator $\Omega$ is defined in
$\left(  \ref{001}-\ref{02}\right)  $ for `nice' states, and in $\left(
\ref{015}-\ref{019}\right)  $ for the general case).

\textbf{1. }The semigroup $S_{t}$ is well-defined, i.e., for every measure
$\nu$ on $M$, the trajectory $S_{t}\left(  \nu\right)  $ exists and is unique.

\textbf{2. }Suppose that the initial states $\nu_{N}$ of the networks
$\mathcal{K}_{N}$ converge to the state $\nu$ of the limiting network
$\mathcal{K}.$ Then for every $t>0$ $S_{N,t}\left(  \nu_{N}\right)
\rightarrow S_{t}\left(  \nu\right)  .$
\end{theorem}

\section{The Non-Linear Markov Process: Existence}

In this section we prove Part 1 of Theorem \ref{main result}.

Our Non-Linear Markovian evolution is a jump process on $M=\Pi_{v}M\left(
v\right)  .$ Between the jumps, the point $q\in M\left(  v\right)  $ moves
with unit speed in its orthant along the field $r\left(  q\right)  $; this
movement is deterministic. The point $q\in M\left(  v\right)  $ can also
perform various jumps, as described above.

\begin{theorem}
For every initial state $\mu\left(  0\right)  =\Pi\mu_{v}\left(  0\right)  $
Equation $\left(  \ref{001}-\ref{02}\right)  $ has a solution, which is unique.
\end{theorem}

\begin{proof}
The idea of the proof is the following. Let us introduce an auxiliary system
on the same set of servers with the same initial condition $\mu\left(
0\right)  .$ Instead of the internal Poisson flows of the initial system with
rates $\bar{\lambda}_{v^{\prime}v}\left(  t\right)  $ (which are
(hypothetically) determined uniquely by $\mu\left(  0\right)  $), we consider,
for every node $v$, an arrival Poisson flow of customers with arbitrary rate
function $\lambda_{v^{\prime}v}\left(  t\right)  .$ The result of the service
at $v$ will then be a collection of (individually non-Poisson) departure flows
to certain nodes $v^{\prime\prime}.$ The flow from $v$ to $v^{\prime\prime}$
is non-Poisson in general. Consider its rate function, $b_{vv^{\prime\prime}%
}\left(  t\right)  :$
\[
b_{vv^{\prime\prime}}\left(  t\right)  =\lim_{\Delta t\rightarrow0}%
\frac{\mathbb{E}\left(  \text{number of cust. arrived from }v\text{ to
}v^{\prime\prime}\text{ in }[t,t+\Delta t]\right)  }{\Delta t}.
\]
Thus, we have an operator $\psi_{\mu\left(  0\right)  },$ which transforms the
collection $\left\{  \lambda_{v^{\prime}v}\left(  t\right)  \right\}  $ to
$\left\{  b_{v^{\prime}v}\left(  t\right)  \right\}  .$ Our theorem about the
existence and uniqueness will follow from the fact that the map $\psi
_{\mu\left(  0\right)  }$ has a unique fixed point, $\bar{\lambda}$. Note,
that the ratefunctions $\lambda$ and $b$ depend not only on the nodes
$v,v^{\prime},$ but also on the class of customers. Below we often omit some
coordinates of these vectors, but we always keep them in mind.

Note that the functions $b$ are continuous and uniformly bounded. Moreover,
without loss of generality we can suppose that they are Lipschitz, with a
Lipschitz constant $L,$ which depends only on $\gamma$ -- the supremum of the
service rates. Since we look for the fixed point, we can assume that the
functions $\lambda$ are bounded as well, and also that they are integrable and
Lipschitz, with the same Lipschitz constant. So we restrict the functions
$\lambda$ to be in this class $L_{L}\left[  0,T\right]  $. In case the graph
$G$ is finite, we put the $L_{1}$ metric on our functions,\textbf{ }%
\[
\int_{0}^{T}\left\vert \lambda^{1}\left(  \tau\right)  -\lambda^{2}\left(
\tau\right)  \right\vert d\tau\equiv\sum_{v^{\prime}v}\int_{0}^{T}\left\vert
\lambda_{v^{\prime}v}^{1}\left(  \tau\right)  -\lambda_{v^{\prime}v}%
^{2}\left(  \tau\right)  \right\vert d\tau.
\]
Note that this metric turns $L_{L}\left[  0,T\right]  $ into complete compact
metric space, by Arzel\`{a}--Ascoli. For the infinite $G$ we choose an
arbitrary vertex $v_{0}\in G$ as a `root', and we define likewise
\begin{align*}
&  \int_{0}^{T}\left\vert \lambda^{1}\left(  \tau\right)  -\lambda^{2}\left(
\tau\right)  \right\vert d\tau\\
&  \equiv\sum_{v^{\prime}v}\exp\left\{  -2D\left(  G\right)  \left[
\mathrm{dist}\left(  v_{0},v^{\prime}\right)  +\mathrm{dist}\left(
v_{0},v\right)  \right]  \right\}  \int_{0}^{T}\left\vert \lambda_{v^{\prime
}v}^{1}\left(  \tau\right)  -\lambda_{v^{\prime}v}^{2}\left(  \tau\right)
\right\vert d\tau.
\end{align*}
\noindent We recall that $D\left(  G\right)  $ is the maximal degree of the
vertex of $G.$ The topology on $L_{L}\left[  0,T\right]  $ thus defined is
equivalent to the Tikhonov topology; in particular, $L_{L}\left[  0,T\right]
$ is again a complete compact metric space.

We will now show that for every $\mu\left(  0\right)  $, the map $\psi
_{\mu\left(  0\right)  }$ is a contraction on $\mathcal{L}_{L}\left[
0,T\right]  $; by the Banach theorem this will imply the existence and the
uniqueness of the fixed point for $\psi_{\mu\left(  0\right)  }$. Without loss
of generality we can assume that $T$ is small.

Let $\left\{  \lambda_{v}^{1}\left(  t\right)  ,\lambda_{v}^{2}\left(
t\right)  :t\in\left[  0,T\right]  ,v\in G\right\}  $ be the rates of two
collections of Poisson inflows to our servers; assume that for all $v$,
\[
\int_{0}^{T}\left\vert \lambda_{v}^{1}\left(  \tau\right)  -\lambda_{v}%
^{2}\left(  \tau\right)  \right\vert d\tau<\Lambda
\]
uniformly in $v.$ We want to estimate the difference $b_{0}^{1}\left(
t\right)  -b_{0}^{2}\left(  t\right)  $ of the rates of the departure flows at
$0\in G;$ clearly, this will be sufficient. The point is that the rates
$b_{0}^{i}\left(  \cdot\right)  $ depend also on the rates $\lambda_{v}%
^{i}\left(  \cdot\right)  $ for $v\neq0$, due to the possibility of servers
jumps, after which the state at the node $0$ is replaced by that at the
neighboring node.

Let $\tau_{1}<\tau_{2}<...<\tau_{k}\in\left[  0,T\right]  ,$ $k=0,1,2,\ldots$
be the (random) moments when the state at $0$ is replaced by the state at a
neighbor node, due to the `server jumps'. We will obtain an estimate on
$\int_{0}^{T}\left\vert b_{0}^{1}\left(  t\right)  -b_{0}^{2}\left(  t\right)
\right\vert dt$ under the condition that the number $k$ and the moments
$\tau_{1}<\tau_{2}<...<\tau_{k}$ are fixed; since our estimate is uniform in
the conditioning, this is sufficient. Note that the probability to have $k$
jumps during the time $T$ is bounded from above by $\left(  T\mathbf{\beta
}\right)  ^{k}$ (see $\left(  \ref{30}\right)  $).

Informally, the contraction takes place because the departure rates $b\left(
t\right)  $ for $t\in\left[  0,T\right]  $ with $T$ small depend mainly on the
initial state $\mu\left(  0\right)  $: the new customers, arriving during the
time $\left[  0,T\right]  $ have little chance to be served before $T,$ if
there were customers already waiting. Therefore the `worst'\ case for us is
when in the initial states all the servers are empty, i.e. the measures
$\mu_{v}\left(  0\right)  $ are all equal to the measure $\delta_{0},$ having
a unit atom at the empty queues $\varnothing.$

\textbf{ A. }Let us start the proof by imposing the condition that the jumps
to the server at $v=0$ do not happen, i.e. $k=0.$ So let $\lambda^{1}\left(
t\right)  ,$ $\lambda^{2}\left(  t\right)  ,t\in\left[  0,T\right]  $ be the
rates of two collections of the Poisson inflows to the empty server, and
$\gamma$ be the supremum of the service rates. We want to estimate the
difference $\int_{0}^{T}\left\vert b^{1}\left(  t\right)  -b^{2}\left(
t\right)  \right\vert dt$ of the rates of the departure flows. For that we
will use the representation of integrals as expectations of random variables.

Consider first the integral
\[
I_{T}^{\lambda}=\sum_{v^{\prime}}\int_{0}^{T}\left\vert \lambda_{v^{\prime}%
v}^{1}\left(  t\right)  -\lambda_{v^{\prime}v}^{2}\left(  t\right)
\right\vert dt.
\]
For every value of the index $v^{\prime}$ let us consider the region between
the graphs of the functions
\[
H_{v^{\prime}}\left(  t\right)  =\max\left\{  \lambda_{v^{\prime}v}^{1}\left(
t\right)  ,\lambda_{v^{\prime}v}^{2}\left(  t\right)  \right\}  \quad
\mbox{and}\quad h_{v^{\prime}}\left(  t\right)  =\min\left\{  \lambda
_{v^{\prime}v}^{1}\left(  t\right)  ,\lambda_{v^{\prime}v}^{2}\left(
t\right)  \right\}  .
\]
Then the integral $\int_{0}^{t}\left\vert \lambda_{v^{\prime}v}^{1}\left(
t\right)  -\lambda_{v^{\prime}v}^{2}\left(  t\right)  \right\vert dt$ is the
expectation of the number of the rate $1$ Poisson points $\omega_{v^{\prime}%
}\in\mathbb{R}^{2},$ falling between $H_{v^{\prime}}$ and $h_{v^{\prime}}.$
Let us denote the corresponding random variables by $\alpha_{v^{\prime}}.$ Let
us declare the customers that correspond to the points falling below
$h_{v^{\prime}}=\min\left\{  \lambda_{v^{\prime}v}^{1}\left(  t\right)
,\lambda_{v^{\prime}v}^{2}\left(  t\right)  \right\}  $ as colorless; let us
color the points falling between $H_{v^{\prime}}$ and $h_{v^{\prime}}$ in the
following way: the customers from the first flow, falling between
$\lambda_{v^{\prime}v}^{1}\left(  t\right)  $ and $h_{v^{\prime}}\left(
t\right)  ,$ get the red color, while the customers from the second flow,
falling between $\lambda_{v^{\prime}v}^{2}\left(  t\right)  $ and
$h_{v^{\prime}}\left(  t\right)  ,$ get the blue one. The variable
$\alpha_{v^{\prime}}$ is thus the number of colored customers. This defines
the coupling between the two input flows.

Note that, on the event that each of the departure flows is not affected by
the colored customers, the two departure flows are identical. Therefore the
contribution of this event to the integral $\int_{0}^{T}\left\vert
b_{vv^{\prime\prime}}^{1}\left(  t\right)  -b_{vv^{\prime\prime}}^{2}\left(
t\right)  \right\vert dt$ is zero. Let $\mathcal{C}$ be the complement to this
event. Note that the probability of arrival of colored customer is equal to
$I_{T}^{\lambda}.$ \textit{Under the condition} of the arrival of the colored
customer $\bar{c}$, the conditional probability that this customer will affect
the departure flow is of the order of $\mathcal{F}T$ (see $\left(
\ref{41}\right)  $). Indeed, either this customer $\bar{c}$ himself will be
served and will departure $v$ during the time $T,$ or else the arrival of
$\bar{c}$ can prevent the departure of the customer with lower priority, whose
service was going to be completed during $\left[  0,T\right]  .$ Hence the
(unconditional) probability of $\mathcal{C}$ is $T\mathcal{F}I_{T}^{\lambda}$
for $T$ small, and therefore its contribution to the integral $\int_{0}%
^{T}\left\vert b_{vv^{\prime\prime}}^{1}\left(  t\right)  -b_{vv^{\prime
\prime}}^{2}\left(  t\right)  \right\vert dt$ is bounded from above by
$const\cdot T\mathcal{F}I_{T}^{\lambda},$ which is $\ll\Lambda$ uniformly.

In the case of non-empty initial queue the situation is even simpler, since we
have more uncolored customers.

\textbf{B. }Consider now the case when we also allow the jumping of servers.
Then the operator $\psi_{\mu\left(  0\right)  }\ $from the collection
$\left\{  \lambda_{e}\left(  t\right)  ,e\in E\left(  G\right)  \right\}  $ to
$\left\{  b_{e}\left(  t\right)  ,e\in E\left(  G\right)  \right\}  ,$ is
still defined, once the initial condition $\mu\left(  0\right)  =\left\{
\mu_{v}\left(  0\right)  \right\}  $ is fixed. We argue that the contribution
of the event of having $k$ server jumps at the node $v=0$ to $\int_{0}%
^{T}\left\vert b_{vv^{\prime\prime}}^{1}\left(  t\right)  -b_{vv^{\prime
\prime}}^{2}\left(  t\right)  \right\vert dt$ is of the order of $T^{k+1}.$

Consider the case $k=1,$ and let us condition on the event that $\tau_{1}%
=\tau<T$ and the jump to $v=0$ is made from the nearest neighbor node (n.n.)
$w$. So we need to compare the evolution at $v=0$ defined by the Poisson
inputs with rates $\lambda_{v^{\prime}0}^{1}\left(  t\right)  ,\lambda
_{w^{\prime}w}^{1}\left(  t\right)  $ with that defined by the rates
$\lambda_{v^{\prime}0}^{2}\left(  t\right)  ,\lambda_{w^{\prime}w}^{2}\left(
t\right)  .$ We will use the same couplings between the pairs of flows. Before
time $\tau$, the server at $v=0$ behaves as described above in Section
\textbf{A}. At the moment $\tau$ its state $\mu_{0}^{i}\left(  \tau\right)  $
is replaced by that drawn independently from $\mu_{w}^{i}\left(  \tau\right)
,$ which then evolves according to the in-flow defined by $\lambda_{v^{\prime
}0}^{i}\left(  t\right)  ,$ $i=1,2.$ A moment thought shows that all the
arguments of Part \textbf{A }still\textbf{ }apply, so under all the conditions
imposed we have that $\int_{0}^{T}\left\vert b_{0v^{\prime\prime}}^{1}\left(
t\right)  -b_{0v^{\prime\prime}}^{2}\left(  t\right)  \right\vert dt$ is again
bounded from above by $const\cdot T\mathcal{F}\max\left(  I_{T}^{\lambda
_{\ast0}},I_{T}^{\lambda_{\ast w}}\right)  $ which is $\ll\Lambda$ uniformly.

The averaging over $\tau$ brings the extra factor $T\mathbf{\beta},$ so the
corresponding contribution to $\int_{0}^{T}\left\vert b_{0v^{\prime\prime}%
}^{1}\left(  t\right)  -b_{0v^{\prime\prime}}^{2}\left(  t\right)  \right\vert
dt$ is $const\cdot T^{2}\mathbf{\beta}\mathcal{F}\max\left(  I_{T}%
^{\lambda_{\ast0}},I_{T}^{\lambda_{\ast w}}\right)  $ $\ll T\Lambda.$

The case of general values of $k$ follows immediately.

\textbf{C.} Thus far we have shown that the initial condition $\mu\left(
0\right)  =\left\{  \mu_{v}\left(  0\right)  \right\}  $ defines all the
in-flow rates $\lambda_{v^{\prime}v}\left(  t\right)  ,$ $t\in\left[
0,T\right]  $ uniquely. If the graph $G$ is finite, that implies the
uniqueness of the evolving measure $\mu\left(  t\right)  .$ For the infinite
graph it can happen in principle that different `boundary conditions' -- i.e.
different evolutions of $\mu$ `at infinity' -- might still be a source of
non-uniqueness. However, the argument of Part \textbf{B} shows that the
influence on the origin $0\in G$ from the nodes at distance $R$ during the
time $T$ is of the order of $T^{R},$ provided the degree of $G$ is bounded.
Therefore the uniqueness holds for the infinite graphs as well.
\end{proof}

\begin{proposition}
\label{Fe} The semigroup, defined by Equations $\left(  \ref{001}%
-\ref{02}\right)  $ is Feller.
\end{proposition}

\begin{proof}
Note that the trajectories $\mu_{t}$ are continuous in time. Indeed, they are
defined uniquely by the inflow rates $\left\{  \bar{\lambda}_{v^{\prime}%
v}\left(  t\right)  \right\}  ,$ which in turn are defined by the external
flows and by the initial state $\mu_{0}.$ Since for the NLMP they coincide
with the departure flow rates, $\left\{  \bar{b}_{v^{\prime}v}\left(
t\right)  \right\}  ,$ which are continuous in time, they are continuous
themselves, and so the trajectories $\mu_{t}$ are continuous as well.
Moreover, the dependence of the rates $\bar{\lambda}_{v^{\prime}v}\left(
t\right)  $ on the initial state $\mu_{0}$ is continuous, since the departure
rates $\left\{  \bar{b}_{v^{\prime}v}\left(  t\right)  \right\}  $ are
continuous in $\mu_{0}.$ Therefore the map $\mathcal{P}\rightarrow
\mathcal{P},$ defined by $\mu_{0}\rightsquigarrow\mu_{t},$ is continuous as well.
\end{proof}

\section{Compactifications}

In order to study the convergence of our mean-field type networks
$\mathcal{K}_{N}$ to the limiting Non-Linear Markov Process, network
$\mathcal{K},$ we want the latter to be defined on a compact state space. This
means that we have to add to the graph $G$ the sites $G_{\infty}$ lying at
infinity, obtaining the extended graph $\bar{G}=G\cup G_{\infty},$ and to
allow infinite queues at each node $v\in\bar{G}.$ We then have to extend our
dynamics to this bigger system. The way it is chosen among several natural
options is of small importance, since, as we will show, if the initial state
of our network assigns zero probability to various infinities, then the same
holds for all times.

The compactification plays here a technical role. It allows us to use some
standard theorems of convergence of Markov processes, and probably can be
avoided. The benefit it brings is that certain observables can be continuously
extended to a larger space.

\subsection{Compactification $\bar{G}$ of the Graph $G$}

When the graph $G$ is infinite, we need to have its compactification. The
compactification we are going to define does depend on our network. Namely, it
will use the following feature of our network discipline: if a customer $c$ is
located at $v$ and its destination $D\left(  c\right)  =w,$ then the path $c$
takes to get from $v$ to $w$ is obtained using the random greedy algorithm.
According to it, $c$ chooses uniformly among all the n.n. sites, which bring
$c$ one unit closer to its goal. In principle, different disciplines on the
same graph $G$ might lead to different compactifications.

To define it we proceed as follows. Let $\gamma=\left\{  \gamma_{n}\in
G\right\}  $ be a n.n. path on $G.$ We want to define a notion of existence of
the limit $L\left(  \gamma\right)  =\lim_{n\rightarrow\infty}\gamma_{n}.$ If
the sequence $\gamma_{n}$ stabilizes, i.e. if $\gamma_{n}\equiv g\in G$ for
all $n$ large enough, we define $L\left(  \gamma\right)  =g.$ To proceed, for
any $v\in G$ we define the Markov chain $P^{v}$ on $G.$ It is a n.n. random
walk, such that at each step a walker makes its distance to $v$ to decrease by
$1.$ If he has several such choices, he choose one of them uniformly.
Therefore the transition probabilities $P^{v}\left(  u,w\right)  $ are given
by the function $e_{u,v}\left(  w\right)  ,$ defined in $\left(
\ref{11}\right)  .$ If $T$ is some integer time moment, and $u$ and $v$ are at
the distance bigger than $T,$ then the $T$-step probability distribution
$P_{u}^{v,T}$ on the trajectories starting at $u$ and heading towards $v$ is defined.

Let now the path $\gamma$ be given, with $\gamma_{n}\rightarrow\infty.$ We say
that the limit $L\left(  \gamma\right)  =\lim_{n\rightarrow\infty}\gamma_{n}$
exists, if for any $u\in G$ and any $T$, the limit $\lim_{n\rightarrow\infty
}P_{u}^{\gamma_{n},T}$ exists. For two paths $\gamma^{\prime},\gamma
^{\prime\prime}$ we say that $L\left(  \gamma^{\prime}\right)  =L\left(
\gamma^{\prime\prime}\right)  $ iff both limits exist and moreover for any
$u\in G$ and any $T$ the measures $P_{u}^{\gamma_{n}^{\prime},T}$ and
$P_{u}^{\gamma_{n}^{\prime\prime},T}$ coincide for all $n\geq n\left(
u,T,\gamma^{\prime},\gamma^{\prime\prime}\right)  $ large enough.

Consider the union $\bar{G}=G\cup G^{\infty}\equiv G\cup\left\{  L\left(
\gamma\right)  :\gamma\text{ is a n.n. path on }G\right\}  .$ It is easy to
see that a natural topology\textbf{ } on $\bar{G}$ makes it into a compact.
For example, consider the case $G=\mathbb{Z}^{2}.$ Let $f:$ $\mathbb{Z}%
^{2}\rightarrow\mathbb{R}^{2}$ be the following embedding: $f\left(
n,m\right)  =\left(  \mathrm{sign}\left(  n\right)  \left(  1-\frac
{1}{\left\vert n\right\vert }\right)  ,\mathrm{sign}\left(  m\right)  \left(
1-\frac{1}{\left\vert m\right\vert }\right)  \right)  .$ Then the closure
$\bar{G}\ $will be the closure of the image of $f$ in $\mathbb{R}^{2}.$

We can make $\bar{G}$ into a graph. To do this we need to specify pairs of
vertices which are connected by an edge. If $v^{\prime},v^{\prime\prime}$ both
belong to $G,$ then they are connected in $\bar{G}$ iff they are connected in
$G.$ If $v^{\prime}$ is in $G,$ while $v^{\prime\prime}\in G^{\infty},$ then
they are never connected. Finally, if $v^{\prime},v^{\prime\prime}\in
G^{\infty},$ then they are connected iff one can find a pair of paths
$\gamma^{\prime},\gamma^{\prime\prime}\rightarrow\infty$ such that $L\left(
\gamma^{\prime}\right)  =v^{\prime},\ L\left(  \gamma^{\prime\prime}\right)
=v^{\prime\prime},$ and the sites $\gamma_{n}^{\prime},\gamma_{n}%
^{\prime\prime}\in G$ are n.n. In particular, every vertex $v\in G^{\infty}$
has a loop attached. Note that the graph $\bar{G}$ is not connected.

\subsection{Extension of the Network to $\bar{G}$}

This is done in a natural straightforward way. Now we have servers and queues
also at `infinite' sites. Note that the customers at infinity cannot get to
the finite part $G$ of $\bar{G}.$ Also, the customers from $G$ cannot get to
$G^{\infty}$ in finite time.

\subsection{Compactification of $M\left(  v\right)  $}

In order to make the manifold $M\left(  v\right)  ,$ $v\in V\left(  G\right)
$ compact we have to add to it various `infinite objects': infinite words,
infinite waiting times and infinite destinations. Since the destinations of
the customers in our network are vertices of the underlying graph, the
compactification of $M\left(  v\right)  $ will depend on it. For the last
question, we use the graph $\bar{G},$ which is already compact. The manifold
of possible queues at $v$ is a disjoint union of the positive orthants
$\mathbb{R}_{w}^{+},$ where $w$ is a finite word, describing the queue
$q_{v}=\left\{  c_{i}\right\}  \equiv\left\{  c_{i}^{v}\right\}  $ at $v.$ The
letters of the word $w$ are classes $\varkappa_{i}\in K$ of the customers plus
the final addresses $v_{i}=D\left(  c_{i}\right)  \in\bar{G}$ of them. So the
set of the values of each letter is compact. We have to compactify the set $W$
of the finite ordered words $w.$

To do it we denote by $O\left(  w\right)  $ the reordering of $w,$ which
corresponds to the order of service of the queue $w.$ This is just a
permutation of $w.$ We say that a sequence $w_{i}$ of finite words converge as
$i\rightarrow\infty,$ iff the sequence of words $O\left(  w_{i}\right)  $
converge coordinate-wise. If this is the case we denote by $\bar{O}%
=\lim_{i\rightarrow\infty}O\left(  w_{i}\right)  ,$ and we say that
$\lim_{i\rightarrow\infty}w_{i}=\bar{O}.$ We denote by $\bar{W}$ the set of
all finite words, $w\equiv\left(  w,O\left(  w\right)  \right)  $,
supplemented by all possible limit points $\bar{O}.$ We define the topology on
$\bar{W}$ by saying that a sequence $w_{i}\in\bar{W}$ is converging iff the
sequence $O\left(  w_{i}\right)  $ converge coordinate-wise. In other words,
we put on $\bar{W}$ the Tikhonov topology. Since $K$ is finite, $\bar{W}$ is
compact in the topology of the pointwise convergence.

According to what was said in the Section \ref{comb}, our service discipline
$\left(  \equiv\text{i.e. the function }O\right)  $ has the following
property. Let the sequence of finite words $w_{i}\in W$ converge in the above
sense. Let $c$ be a customer, and consider the new sequence $w_{i}\cup c\in
W,$ where customer $c$ is the last arrived. Then we have the implication%
\[
\lim_{i\rightarrow\infty}O\left(  w_{i}\right)  \text{ exists }\Rightarrow
\lim_{i\rightarrow\infty}O\left(  w_{i}\cup c\right)  \text{ exists.}%
\]

The continuity of the transition probabilities in this topology on the set of
queues is easy to see; indeed, the closeness of the two queues $q$ and
$q^{\prime}$ means that the first $k$ customers served in both of them are the
same. But then the transition probabilities $P_{T}\left(  q,\cdot\right)  $
and $P_{T}\left(  q^{\prime},\cdot\right)  $ differ by $o\left(  T^{k}\right)
.$ So the extended process is Feller, as well as the initial one.

The compactifications $\mathbb{\bar{R}}_{w}^{+}$ of the orthants
$\mathbb{R}_{w}^{+}$ are defined in the obvious way: they are the products of
$\left\vert w\right\vert $ copies of the compactifications $\mathbb{\bar{R}%
}^{+}=\mathbb{R}^{+}\cup\infty.$ For the infinite words $\bar{O}$ we consider
the infinite products, in the Tikhonov topology. The notion of the convergence
in the union $\cup_{w\in\bar{W}}\mathbb{\bar{R}}_{w}^{+}$ is that of the
coordinate-wise convergence.

The properties of the service times, formulated in Sect. \ref{serv}, allow us
to extend the relevant rates in a continuous way to a function on $\tau
\in\mathbb{\bar{R}}^{+}.$ Moreover, the analog of Proposition \ref{Fe} holds.

\section{\label{conv} The proof of Convergence}

Let $\Omega_{N},\Omega:X\rightarrow X$ are (unbounded) operators on the Banach
space $X.$ We are looking for the conditions of convergence of the semigroups
$\exp\left\{  t\Omega_{N}\right\}  \rightarrow\exp\left\{  t\Omega\right\}  $
on $X$ as $N\rightarrow\infty.$

We will use the following version of the Trotter-Kurtz theorem, which is
Theorem 6.1 from Chapter 1 of \cite{EK}, with the core characterization taken
from Proposition 3.3 of the same Chapter 1, where a \textbf{core} of $\Omega$
is a dense subspace $\bar{X}\subset X$ such that $\forall\psi\in\bar{X}$
\begin{equation}
\exp\left\{  t\Omega\right\}  \left(  \psi\right)  \in\bar{X}. \label{013}%
\end{equation}

\begin{theorem}
\label{TK} Let $X^{1}\subset X^{2}\subset...\subset X$ be a sequence of
subspaces of Banach space $X.$ We suppose that we have projectors $\pi
^{N}:X\rightarrow X^{N},$ such that for every $f\in X$ we have $f^{N}=\pi
^{N}\left(  f\right)  \rightarrow f,$ as $N\rightarrow\infty.$ Suppose we have
(strongly continuous contraction) semigroups $\exp\left\{  t\Omega
_{N}\right\}  $ on $X^{N},$ and we want a condition ensuring that for every
$f$
\begin{equation}
\exp\left\{  t\Omega_{N}\right\}  f^{N}\rightarrow\exp\left\{  t\Omega
\right\}  f. \label{011}%
\end{equation}
For this it is sufficient that for every $f$ in the \textit{core} $\bar{X}$ of
$\Omega$ we have
\begin{equation}
\left\Vert \Omega_{N}\left(  f^{N}\right)  -\Omega\left(  f^{N}\right)
\right\Vert \rightarrow0. \label{012}%
\end{equation}

\end{theorem}

We are going to apply it to our situation. The main ideas of this application
were developed in the paper \cite{KR}.

The strong continuity of the semigroups $\exp\left\{  t\Omega_{N}\right\}  $
is straightforward, while that for the semigroup $\exp\left\{  t\Omega
\right\}  $ follows from the fact that the trajectories $S_{t}\mu$ are
continuous in $t$. After compactification, the set of probability measures
$\mu\in\mathcal{P}$ becomes a compact, $\mathcal{\bar{P}},$ and so the family
$S_{t}\mu$ is equicontinuous. This implies the strong continuity.

It is clear that the operation $\Omega_{N}\left(  f^{N}\right)  $ consists of
computing finite differences for the function $f$ at some atomic measures,
while $\Omega\left(  f^{N}\right)  $ is the operation of computing the
derivatives at the same atomic measures. We will have convergence $\left(
\ref{012}\right)  $ in the situation when the derivatives exists and can be
approximated by the finite differences. Therefore we have to choose
differentiable functions for our space $\bar{X}.$ We will check $\left(
\ref{012}\right)  $ in the Section \ref{finite}.

We then need to check that this space of differentiable functions is preserved
by the semigroup $\exp\left\{  t\Omega\right\}  $ (core property). The
function $\exp\left\{  t\Omega\right\}  f\left(  \mu\right)  $ is just
$f\left(  S_{t}\mu\right)  ,$ so if $f$ is differentiable in $\mu$, then the
differentiability of $f\left(  S_{t}\mu\right)  $ follows from that of
$S_{t}\mu.$

We take for our function space $X$ the space $\mathcal{C}_{u}\left(
\mathcal{\bar{P}}\right)  ,$ with $\mathcal{\bar{P}}$ being the compactified
version of $\mathcal{P},$ the space of product probability measures on $\Pi
M\left(  v\right)  .$ So we consider continuous functions on $\mathcal{P},$
which are bounded and uniformly continuous. $\mathcal{P}$ in turn is a
subspace in the space $\Pi\mathcal{C}^{\ast}\left(  M\left(  v\right)
\right)  $ of linear functionals on $\Pi\mathcal{C}\left(  M\left(  v\right)
\right)  $ (with weak convergence topology), while $\mathcal{C}_{u}\left(
M\left(  v\right)  \right)  $ is the space of bounded uniformly continuous
functions on the combs $M\left(  v\right)  .$ However, we will use a different
norm on $\mathcal{P}.$ Namely, we consider the space $\mathcal{C}_{1}\left(
M\left(  v\right)  \right)  $ of functions $f$ which are continuous and have
continuous derivatives $f^{\prime}$, and we put $\left\vert \left\vert
f\right\vert \right\vert _{1}=\left\vert \left\vert f\right\vert \right\vert
+\left\vert \left\vert f^{\prime}\right\vert \right\vert .$ The space
$\mathcal{P}$ belongs to the dual space $\Pi\mathcal{C}_{1}^{\ast}\left(
M\left(  v\right)  \right)  ,$ and we will define the norm $\left\vert
\left\vert \cdot\right\vert \right\vert _{1}$on $\mathcal{P}$ to be the
restriction of the natural norm on $\Pi\mathcal{C}_{1}^{\ast}\left(  M\left(
v\right)  \right)  .$

To estimate the norm of the Frechet differential we have to consider a
starting measure $\mu^{1}=\Pi\mu_{v}\left(  t=0\right)  ,$ its perturbation,
$\mu^{2}=\Pi\left(  \mu_{v}+h_{v}\right)  ,$ with $h=\left\{  h_{v}\right\}  $
non-trivial for finitely many $v$-s (say), the perturbation having norm
$\approx\sum_{v}\left\vert \left\vert h_{v}\right\vert \right\vert ,$ then to
take the difference $\mu^{2}\left(  T\right)  -\mu^{1}\left(  T\right)  $ and
to write it as $\Phi\left(  T,\mu^{1}\right)  h+O\left(  \left\vert \left\vert
h\right\vert \right\vert ^{2}\right)  ,$ and finally to show that the norm of
the operator $\Phi\left(  T,\mu^{1}\right)  $ is finite.

Note first that it is enough to prove this for $T$ small, the smallness being
uniform in all relevant parameters.

Now, if the increment $h$ is small (even only in the (weaker) $\left\vert
\left\vert \cdot\right\vert \right\vert _{1}$ sense, i.e. only from the point
of view of the smooth functions -- like, for example, a small shift of a
$\delta$-measure), then all the flows in our network, started from $\mu^{1}%
\ $and $\mu^{2}$ differ only a little, during a short time, its shortness
being a function of the service time distributions only. The amplitude of the
flow difference is of the order of $D\times\mathcal{F}\times\left\vert
\left\vert h\right\vert \right\vert _{1},$ where $D$ is the maximal degree of
the graph $G,$ and $\mathcal{F}$ is defined by $\left(  \ref{41}\right)  .$ In
fact, it can be smaller; it is attained for the situations when a server,
being empty in the state $\mu^{1},$ becomes non-empty in the $h-$perturbed
state. The order is computed in the stronger $\left\vert \left\vert
\cdot\right\vert \right\vert _{0}$ norm. Therefore, after time $T$ the norm of
the difference is such that
\begin{equation}
\left\vert \left\vert \mu^{2}\left(  T\right)  -\mu^{1}\left(  T\right)
\right\vert \right\vert \leq\left\vert \left\vert h\right\vert \right\vert
+T\times D\times\mathcal{F}\times\left\vert \left\vert h\right\vert
\right\vert \ \mathbf{,} \label{19}%
\end{equation}
which explains our claim about the operator $\Phi\left(  T,\mu^{1}\right)  .$
We will give a formal proof in Section \ref{FD}.

\subsection{\label{finite} Generator Comparison}

We start with the finite MF-type network, made from $N$ copies of the initial network.

We first describe the process as a process of the network containing
$N\left\vert G\right\vert $ servers, and then do the factorization by the
product of $\left\vert G\right\vert $ permutation groups $S_{N}.$

The former one will be described only briefly. At each of the $N\left\vert
G\right\vert $ servers, there is a queue of customers. Some of the queues can
be empty. As time goes, the queues evolve -- due to (1) the arrivals of
external customers; (2) the end of the service of a customer, which then
leaves the network; (3) the end of the service of a customer, which then jumps
to the next server; (4) the interchange of two servers. Each of these events
leads to a jump of our process. If none of them happens, the process evolves
continuously, with the help of the time-shift semigroup. 

After the factorization by the permutation group $\left(  \mathcal{S}%
_{N}\right)  ^{G}$ the configuration at any vertex $v\in G$ can be
conveniently described by an atomic probability measure $\Delta_{N}^{v}$ on
$M\left(  v\right)  $ of the form $\sum_{k=1}^{N}\frac{1}{N}\delta\left(
q_{v,k},\tau\right)  ,$ where $\tau$ is the vector of time durations some of
the customers from the queue $q_{v,k}$ (in particular the customer $C\left(
q_{v,k}\right)  $) were already under the service; the case of empty queue is included.

We have a semigroup $S_{N}$\textbf{ }and its generator $\Omega_{N};$ the
existence of them is straightforward. $\Omega_{N}$ acts on functions $F$ on
measures $\mu_{N}\in\mathcal{P}_{N},$ which are atomic measures with atoms of
weight $\frac{1}{N}.$

Our goal is now the following. Let $\mu_{N}\rightarrow\mu,$ and $F$ be a
smooth function on measures. Let us look at the limit $\Omega_{N}\left(
F\right)  \left(  \mu_{N}\right)  $ and the value $\Omega\left(  F\right)
\left(  \mu\right)  .$ Since $F$ is smooth,\textbf{ }we are able replace
certain differences by derivatives; after this, we will see the convergence
$\Omega_{N}\left(  F\right)  \left(  \mu_{N}\right)  \rightarrow\Omega\left(
F\right)  \left(  \mu\right)  $ in a transparent way.

For this, we write the multiline formula $\left(  \ref{031}-\ref{034}\right)
$ for the operator $\Omega_{N},$ applied to a function $F.$ On each $M\left(
v\right)  $, we have to take a probability measure $\Delta_{N}^{v}$ of the
form $\sum_{k=1}^{N}\frac{1}{N}\delta\left(  q_{v,k},\tau\right)  ,$ where
$\tau$ is the amount of service already received by customer $C\left(
q_{v,k}\right)  $ of $q_{v,k}$ (at the position $i^{\ast}\left(
q_{v,k}\right)  $ in queue $q_{v,k}$). Let $\Delta_{N}=\left\{  \Delta_{N}%
^{v}\right\}  .$ Then%
\begin{align}
&  \left(  \Omega_{N}\left(  F\right)  \right)  \left(  \Delta_{N}\right)
\nonumber\\
&  =\sum_{v}\sum_{k}\frac{\partial F}{\partial r\left(  q_{v,k},\tau\left(
C\left(  q_{v,k}\right)  \right)  \right)  }\left(  \Delta_{N}\right)
\ \text{ }\label{031}\\
&  +\sum_{v}\sum_{k}\sum_{v^{\prime}\text{n.n.}v}\sum_{k^{\prime}}\frac{1}%
{N}e_{v,D\left(  C\left(  q_{v,k}\right)  \right)  }\left(  v^{\prime}\right)
\sigma_{f}\left(  q_{v,k},q_{v,k}\ominus C\left(  q_{v,k}\right)  \right)
\times\nonumber\\
&  \times\left[  F\left(  J_{v,v^{\prime};k,k^{\prime}}\left(  \Delta
_{N}\right)  \right)  -F\left(  \Delta_{N}\right)  \right]
\end{align}
where, for a directed edge $v,v^{\prime}$ and for a pair of queues $q_{v,k},$
$q_{v^{\prime},k^{\prime}}$, we denote by $J_{v,v^{\prime};k,k^{\prime}%
}\left(  \Delta_{N}\right)  $ a new atomic measure, which is a result of the
completion of the service of a customer $C\left(  q_{v,k}\right)  $ in a queue
$q_{v,k}$ at the server $v$ and its subsequent jump into the queue
$q_{v^{\prime},k^{\prime}},$ increasing thereby the length of the queue
$q_{v^{\prime},k^{\prime}}$ at $v^{\prime}$ by one;%

\begin{align}
&  +\sum_{v}\sum_{k}\sum_{v^{\prime}}\sum_{\varkappa}\lambda\left(
\varkappa,v,v^{\prime}\right)  \times\label{032}\\
&  \times\left[  F\left(  \Delta_{N}-\frac{1}{N}\delta\left(  q_{v,k}\right)
+\frac{1}{N}\delta\left(  q_{v,k}\oplus c^{v}\left(  \varkappa,0,v^{\prime
}\right)  \right)  \right)  -F\left(  \Delta_{N}\right)  \right]
\end{align}
-- here we see the arrival of a new customer of class $\varkappa$ with a
destination $v^{\prime};$%
\begin{align}
&  +\sum_{v}\sum_{k}\sum_{v^{\prime}:\mathrm{dist}\left(  v,v^{\prime}\right)
\leq1}\delta\left(  D\left(  C\left(  q_{v,k}\right)  \right)  =v^{\prime
}\right)  \sigma_{f}\left(  q_{v,k},q_{v,k}\ominus C\left(  q_{v,k}\right)
\right)  \times\label{033}\\
&  \times\left[  F\left(  J_{v;k}\left(  \Delta_{N}\right)  \right)  -F\left(
\Delta_{N}\right)  \right]
\end{align}
-- here we account for the customers that leave the network; the operator
$J_{v;k}\left(  \Delta_{N}\right)  $ denotes the new atomic measure, which is
a result of the completion of the service of a customer $C\left(
q_{v,k}\right)  $ in a queue $q_{v,k}$ at the server $v$ and its subsequent
exit from the system: $J_{v;k}\left(  \Delta_{N}\right)  =\Delta_{N}-\frac
{1}{N}\delta\left(  q_{v,k}\right)  +\frac{1}{N}\delta\left(  q_{v,k}%
q_{v,k}\ominus C\left(  q_{v,k}\right)  \right)  $;
\begin{equation}
+\sum_{v}\sum_{k}\sum_{v^{\prime}\text{n.n.}v}\sum_{k^{\prime}}\frac{1}%
{N}\beta_{vv^{\prime}}\left[  F\left(  T_{v,k;v^{\prime},k^{\prime}}\Delta
_{N}\right)  -F\left(  \Delta_{N}\right)  \right]  \label{034}%
\end{equation}
-- here the operator $T_{v,k;v^{\prime},k^{\prime}}$ acts on the measure
$\Delta_{N}$ by exchanging the atoms $\frac{1}{N}\delta\left(  q_{v,k}\right)
$ and $\frac{1}{N}\delta\left(  q_{v^{\prime},k^{\prime}}\right)  .$

\textbf{Remark. }If our graph is infinite, the sums in $\left(  \ref{031}%
-\ref{034}\right)  $ are infinite. However, they make sense for \textit{local
}functions $F,$ as well as for the \textit{quasilocal }ones, which depend of
far-away nodes (exponentially) weakly.

Next, let us pass in the above formula to a formal limit, obtaining thus the
(formal) expression for the limiting operator $\Omega.$ It acts on functions
$F$ on all probability measures on $\cup_{v}M\left(  v\right)  .$ Let
$\Delta=\lim_{N\rightarrow\infty}\Delta_{N}.$ We will need the
differentiability of $F$ at $\Delta,$ i.e. the existence of the Frechet
differential $F^{\prime}$ at $\Delta.$ This differential will be denoted also
by $F_{\Delta}^{\prime}\left(  \cdot\right)  ;$ it is a linear functional on
the space of tangent vectors to $\mathcal{P}$ at $\Delta\in\mathcal{P}$.
Assuming the existence of this differential, we have another multiline
expression $\left(  \ref{015}-\ref{019}\right)  :$%
\begin{align}
\left(  \Omega\left(  F\right)  \right)  \left(  \Delta\right)   &  =\left(
\hat{\sigma}\left(  F\right)  \right)  \left(  \Delta\right) \label{015}\\
&  +\sum_{v}\sum_{v^{\prime}\text{n.n.}v}F_{\Delta}^{\prime}e_{v,D\left(
C\left(  \cdot\right)  \right)  }\left(  v^{\prime}\right)  \sigma_{f}%
\times\left(  \zeta_{vv^{\prime}}\left(  \Delta\right)  -\Delta\right)
\nonumber
\end{align}
-- where $\zeta_{vv^{\prime}}\left(  \Delta\right)  -\Delta$ is a (signed)
measure (see $\left(  \ref{53}\right)  $), and where $e_{v,D\left(  C\left(
\cdot\right)  \right)  }\left(  v^{\prime}\right)  \sigma_{f}\times\left(
\zeta_{vv^{\prime}}\left(  \Delta\right)  -\Delta\right)  $ denotes the
measure having density

\noindent$e_{v,D\left(  C\left(  q\right)  \right)  }\left(  v^{\prime
}\right)  \sigma_{f}\left(  q_{v},q_{v}\ominus C\left(  q_{v}\right)  \right)
$ with respect to $\left(  \zeta_{vv^{\prime}}\left(  \Delta\right)
-\Delta\right)  \left(  dq\right)  ;$%
\begin{equation}
+\sum_{v}\sum_{v^{\prime}}\sum_{\varkappa}\lambda\left(  \varkappa
,v,v^{\prime}\right)  F_{\Delta}^{\prime}\left(  \chi_{v,v^{\prime};\varkappa
}\left(  \Delta\right)  -\Delta\right)  \label{016}%
\end{equation}
-- here\textbf{ }$\chi_{v,v^{\prime};\varkappa}:M\left(  v\right)  \rightarrow
M\left(  v\right)  $ is the embedding, corresponding to the arrival to $v$ of
the external customer of class $\varkappa$ and destination $v^{\prime},$ see
$\left(  \ref{51}\right)  ;$%
\begin{equation}
+\sum_{v}F_{\Delta}^{\prime}\left(  \sigma_{f}\times\left(  \psi_{nn}%
^{v}\left(  \Delta\right)  -\Delta\right)  \right)  \label{017}%
\end{equation}
-- here
\begin{equation}
\psi_{nn}^{v}\left(  q\right)  =\left\{
\begin{array}
[c]{cc}%
\psi^{v}\left(  q\right)  & \text{ for }q\text{ with }\mathrm{dist}\left(
v,D\left(  C\left(  q\right)  \right)  \right)  \leq1\\
q & \text{ for }q\text{ with }\mathrm{dist}\left(  v,D\left(  C\left(
q\right)  \right)  \right)  >1
\end{array}
\right.  , \label{018}%
\end{equation}
while $\psi^{v}:M\left(  v\right)  \rightarrow M\left(  v\right)  $ is the
projection, see $\left(  \ref{52}\right)  ,$ and the term $\sigma_{f}%
\times\left(  \psi_{nn}^{v}\left(  \Delta\right)  -\Delta\right)  $ is the
(signed-)measure, having density $\sigma_{f}\left(  q\right)  $ with respect
to the measure $\left(  \psi_{nn}^{v}\left(  \Delta\right)  -\Delta\right)
\left(  dq\right)  ;$%

\begin{equation}
+\sum_{v}\sum_{v^{\prime}\text{n.n.}v}\beta_{v^{\prime}v}F_{\Delta}^{\prime
}\left(  \left(  T_{v^{\prime}v}\Delta\right)  -\Delta\right)  ; \label{019}%
\end{equation}
-- here the operator $T_{v^{\prime}v}$ acts on the measure $\Delta$ in the
following way: it replaces the component $\Delta_{v}$ of the measure $\Delta$
by the measure $\Delta_{v^{\prime}}$ (via identification between $M\left(
v\right)  $ and $M\left(  v^{\prime}\right)  $).

We now check that the limiting operator $\Omega$ is the same one we were
dealing with in our study of the non-linear Markov process, $\left(
\ref{001}-\ref{02}\right)  .$

\begin{proposition}
\label{mu} The formula $\left(  \ref{015}-\ref{019}\right)  $ can be written
in the form
\begin{equation}
\left(  \Omega\left(  F\right)  \right)  \left(  \mu\right)  =\left(
\hat{\sigma}\left(  F\right)  \right)  \left(  \mu\right)  +\left(  F^{\prime
}\left(  \mu\right)  \right)  \left(  g\left(  \mu\right)  \right)  ,
\label{130}%
\end{equation}
where $\hat{\sigma}$ is the generator of the time\textbf{-}shift semigroup,
acting on our manifold, $F^{\prime}$ is the Frechet differential of $F$, and
the (signed) measure $g\left(  \mu\right)  $ is given by the r.h.s. of
$\left(  \ref{001}-\ref{02}\right)  .$
\end{proposition}

\begin{proof}
For the convenience of the reader we repeat here the equation $\left(
\ref{001}-\ref{02}\right)  .$
\begin{align*}
&  \frac{d}{dt}\mu_{v}\left(  q_{v},t\right) \\
&  =-\frac{d}{dr_{i^{\ast}\left(  q_{v}\right)  }\left(  q_{v}\right)  }%
\mu_{v}\left(  q_{v},t\right) \\
&  +\delta\left(  0,\tau\left(  e\left(  q_{v}\right)  \right)  \right)
\mu_{v}\left(  q_{v}\ominus e\left(  q_{v}\right)  \right)  \left[
\sigma_{tr}\left(  q_{v}\ominus e\left(  q_{v}\right)  ,q_{v}\right)
+\sigma_{e}\left(  q_{v}\ominus e\left(  q_{v}\right)  ,q_{v}\right)  \right]
\\
&  -\mu_{v}\left(  q_{v},t\right)  \sum_{q_{v}^{\prime}}\left[  \sigma
_{tr}\left(  q_{v},q_{v}^{\prime}\right)  +\sigma_{e}\left(  q_{v}%
,q_{v}^{\prime}\right)  \right] \\
&  +\left[  \int_{q_{v}^{\prime}:q_{v}^{\prime}\ominus C\left(  q_{v}^{\prime
}\right)  =q_{v}}d\mu_{v}\left(  q_{v}^{\prime}\right)  \sigma_{f}\left(
q_{v}^{\prime},q_{v}^{\prime}\ominus C\left(  q_{v}^{\prime}\right)  \right)
\right]  -\mu_{v}\left(  q_{v}\right)  \sigma_{f}\left(  q_{v},q_{v}\ominus
C\left(  q_{v}\right)  \right) \\
&  +\sum_{v^{\prime}\text{n.n.}v}\beta_{vv^{\prime}}\left[  \mu_{v^{\prime}%
}\left(  q_{v}\right)  -\mu_{v}\left(  q_{v}\right)  \right]  ~.
\end{align*}

\noindent The term $\left(  \hat{\sigma}\left(  F\right)  \right)  \left(
\Delta\right)  $ evidently corresponds to $-\frac{d}{dr_{i^{\ast}\left(
q_{v}\right)  }\left(  q_{v}\right)  }\mu_{v}\left(  q_{v},t\right)  ,$ and
the term
\[
\sum_{v}\sum_{v^{\prime}\text{n.n.}v}\beta_{v^{\prime}v}F_{\Delta}^{\prime
}\left(  \left(  T_{v^{\prime}v}\Delta\right)  -\Delta\right)
\]
-- to $\sum_{v^{\prime}\text{n.n.}v}\beta_{vv^{\prime}}\left[  \mu_{v^{\prime
}}\left(  q_{v}\right)  -\mu_{v}\left(  q_{v}\right)  \right]  .$ The
`external customer arrival' term
\[
\sum_{v}\sum_{v^{\prime}}\sum_{\varkappa}\lambda\left(  \varkappa,v,v^{\prime
}\right)  F_{\Delta}^{\prime}\left(  \chi_{v,v^{\prime};\varkappa}\left(
\Delta\right)  -\Delta\right)
\]
matches the terms
\[
\delta\left(  0,\tau\left(  e\left(  q_{v}\right)  \right)  \right)  \mu
_{v}\left(  q_{v}\ominus e\left(  q_{v}\right)  \right)  \sigma_{e}\left(
q_{v}\ominus e\left(  q_{v}\right)  ,q_{v}\right)  -\mu_{v}\left(
q_{v},t\right)  \sum_{q_{v}^{\prime}}\sigma_{e}\left(  q_{v},q_{v}^{\prime
}\right)  .
\]
The `intermediate service completion' term
\[
\sum_{v}\sum_{v^{\prime}\text{n.n.}v}F_{\Delta}^{\prime}\left(  e_{v,D\left(
C\left(  \cdot\right)  \right)  }\left(  v^{\prime}\right)  \sigma_{f}%
\times\left(  \zeta_{vv^{\prime}}\left(  \Delta\right)  -\Delta\right)
\right)
\]
matches the terms%
\[
\delta\left(  0,\tau\left(  e\left(  q_{v}\right)  \right)  \right)  \mu
_{v}\left(  q_{v}\ominus e\left(  q_{v}\right)  \right)  \sigma_{tr}\left(
q_{v}\ominus e\left(  q_{v}\right)  ,q_{v}\right)  -\mu_{v}\left(
q_{v},t\right)  \sum_{q_{v}^{\prime}}\sigma_{tr}\left(  q_{v},q_{v}^{\prime
}\right)  .
\]
Finally, the `final service completion' term $\sum_{v}F_{\Delta}^{\prime
}\left(  \sigma_{f}\times\left(  \psi_{nn}^{v}\left(  \Delta\right)
-\Delta\right)  \right)  $ matches
\[
\left[  \int_{q_{v}^{\prime}:q_{v}^{\prime}\ominus C\left(  q_{v}^{\prime
}\right)  =q_{v}}d\mu_{v}\left(  q_{v}^{\prime}\right)  \sigma_{f}\left(
q_{v}^{\prime},q_{v}^{\prime}\ominus C\left(  q_{v}^{\prime}\right)  \right)
\right]  -\mu_{v}\left(  q_{v}\right)  \sigma_{f}\left(  q_{v},q_{v}\ominus
C\left(  q_{v}\right)  \right)  .
\]

\end{proof}

Let us check that indeed we have the norm-convergence of the operators
$\Omega_{N}$ to $\Omega,$ the one needed in the convergence statement,
$\left(  \ref{012}\right)  $. The norm we use here is again $\left\Vert
\cdot\right\Vert _{1}.$

The precise statement we need is the following:

\begin{proposition}
Let $f$ be a function on $\mathcal{P},$ with $\left\Vert f\right\Vert _{1}$
finite. We can restrict $f$ on each subspace $\mathcal{P}_{N},$ and then apply
the operator $\Omega_{N},$ thus getting a function $\Omega_{N}f$ on
$\mathcal{P}_{N}.$ We can also restrict the function $\Omega f$ from
$\mathcal{P}$ to $\mathcal{P}_{N}.$ Then
\[
\left\Vert \Omega_{N}f-\Omega f\right\Vert _{1}^{\mathcal{P}_{N}}\leq
C_{N}\left\Vert f\right\Vert _{1},
\]
with $C_{N}\rightarrow0,$ where $\left\Vert \cdot\right\Vert _{1}%
^{\mathcal{P}_{N}}$ is the restriction of the norm $\left\Vert \cdot
\right\Vert _{1}$ to the subspace of functions of the measures $\mathcal{P}%
_{N}.$
\end{proposition}

\begin{proof}
We have to compare the operators given by $\left(  \ref{031}-\ref{034}\right)
$ and $\left(  \ref{015}-\ref{019}\right)  $ term by term. For example,
compare the term $\left(  \ref{032}\right)  $
\[
\sum_{v}\sum_{k}\sum_{v^{\prime}}\sum_{\varkappa}\lambda\left(  \varkappa
,v,v^{\prime}\right)  \left[  F\left(  \Delta_{N}-\frac{1}{N}\delta\left(
q_{v,k}\right)  +\frac{1}{N}\delta\left(  q_{v,k}\oplus c^{v}\left(
\varkappa,0,v^{\prime}\right)  \right)  \right)  -F\left(  \Delta_{N}\right)
\right]  ,
\]
which corresponds to the arrival of a new customer of class $\varkappa$ with a
destination $v^{\prime};$ and the term $\left(  \ref{016}\right)  $
\[
\sum_{v}\sum_{v^{\prime}}\sum_{\varkappa}\lambda\left(  \varkappa,v,v^{\prime
}\right)  F_{\Delta_{N}}^{\prime}\left(  \chi_{v,v^{\prime};\varkappa}\left(
\Delta_{N}\right)  -\Delta_{N}\right)  ,
\]
where\textbf{ }$\chi_{v,v^{\prime};\varkappa}:M\left(  v\right)  \rightarrow
M\left(  v\right)  $ is the embedding, corresponding to the arrival to $v$ of
the external customer of class $\varkappa$ and destination $v^{\prime}.$ Due
to the locality properties of $F$ it is sufficient to establish the
convergence:%
\begin{align}
&  \sum_{k=1}^{N}\left[  F\left(  \Delta_{N}-\frac{1}{N}\delta\left(
q_{v,k}\right)  +\frac{1}{N}\delta\left(  q_{v,k}\oplus c^{v}\left(
\varkappa,0,v^{\prime}\right)  \right)  \right)  -F\left(  \Delta_{N}\right)
\right] \label{0021}\\
&  \rightarrow F_{\Delta_{N}}^{\prime}\left(  \chi_{v,v^{\prime};\varkappa
}\left(  \Delta_{N}\right)  -\Delta_{N}\right)
\end{align}
as $N\rightarrow\infty.$

The measure $\Delta_{N}$ is a collection of $N$ atoms, corresponding to queues
$q_{v,k},$ $k=1,...,N.$ In the first expression, we change just one of the $N$
atoms, adding a new customer $c^{v}\left(  \varkappa,0,v^{\prime}\right)  $ to
each of the queues $q_{v,k},$ and then take a sum of the corresponding
increments over $k.$ In the second expression, we change all atoms
simultaneously, obtaining the measure $\chi_{v,v^{\prime};\varkappa}\left(
\Delta_{N}\right)  ,$ plus instead of taking the increment $F\left(
\chi_{v,v^{\prime};\varkappa}\left(  \Delta_{N}\right)  \right)  -F\left(
\Delta_{N}\right)  $ we take the differential $F_{\Delta_{N}}^{\prime}$ of the
measure $\chi_{v,v^{\prime};\varkappa}\left(  \Delta_{N}\right)  -\Delta_{N}.$
To see the norm convergence in $\left(  \ref{0021}\right)  $ let us rewrite
the increments
\begin{align*}
&  F\left(  \Delta_{N}-\frac{1}{N}\delta\left(  q_{v,k}\right)  +\frac{1}%
{N}\delta\left(  q_{v,k}\oplus c^{v}\left(  \varkappa,0,v^{\prime}\right)
\right)  \right)  -F\left(  \Delta_{N}\right) \\
&  =F_{Q_{k,N}}^{\prime}\left(  -\frac{1}{N}\delta\left(  q_{v,k}\right)
+\frac{1}{N}\delta\left(  q_{v,k}\oplus c^{v}\left(  \varkappa,0,v^{\prime
}\right)  \right)  \right)  ,
\end{align*}
by the intermediate value theorem. Here the points $Q_{k,N}$ are some points
on the segments
\[
\left[  \Delta_{N},\Delta_{N}-\frac{1}{N}\delta\left(  q_{v,k}\right)
+\frac{1}{N}\delta\left(  q_{v,k}\oplus c^{v}\left(  \varkappa,0,v^{\prime
}\right)  \right)  \right]  .
\]
Note that the norms $\left\Vert Q_{k,N}-\Delta_{N}\right\Vert $ evidently go
to zero as $N\rightarrow\infty,$ and so $\left\Vert F_{Q_{k,N}}^{\prime
}-F_{\Delta_{N}}^{\prime}\right\Vert \rightarrow0$ as well. Thus%
\begin{align*}
&  \sum_{k=1}^{N}F_{Q_{k,N}}^{\prime}\left(  -\frac{1}{N}\delta\left(
q_{v,k}\right)  +\frac{1}{N}\delta\left(  q_{v,k}\oplus c^{v}\left(
\varkappa,0,v^{\prime}\right)  \right)  \right) \\
&  =F_{\Delta_{N}}^{\prime}\left[  \sum_{k=1}^{N}\left(  -\frac{1}{N}%
\delta\left(  q_{v,k}\right)  +\frac{1}{N}\delta\left(  q_{v,k}\oplus
c^{v}\left(  \varkappa,0,v^{\prime}\right)  \right)  \right)  \right] \\
&  +\sum_{k=1}^{N}\left(  F_{Q_{k,N}}^{\prime}-F_{\Delta_{N}}^{\prime}\right)
\left(  -\frac{1}{N}\delta\left(  q_{v,k}\right)  +\frac{1}{N}\delta\left(
q_{v,k}\oplus c^{v}\left(  \varkappa,0,v^{\prime}\right)  \right)  \right)  ,
\end{align*}
with the second term is uniformly small as $N\rightarrow\infty$. By
definition,
\[
\chi_{v,v^{\prime};\varkappa}\left(  \Delta_{N}\right)  -\Delta_{N}=\sum
_{k=1}^{N}\left(  -\frac{1}{N}\delta\left(  q_{v,k}\right)  +\frac{1}{N}%
\delta\left(  q_{v,k}\oplus c^{v}\left(  \varkappa,0,v^{\prime}\right)
\right)  \right)  ,
\]
and this proves the convergence needed. The other terms are compared in the
same manner.
\end{proof}

This completes the checking of the relation $\left(  \ref{012}\right)  $ of
the Convergence Theorem \ref{TK}.

\subsection{\label{FD} Frechet Differential Properties}

\begin{proposition}
\label{diff} The semigroup is uniformly differentiable in $t.$ In the
notations of Sect. $\ref{conv}$ (see $\left(  \ref{19}\right)  $) this means
that for all $t<T$
\[
\left\vert \left\vert \mu^{2}\left(  t\right)  -\mu^{1}\left(  t\right)
-\left[  \mathcal{D}\mu^{1}\left(  t\right)  \right]  \left(  h\right)
\right\vert \right\vert _{1}\leq\left\vert \left\vert h\right\vert \right\vert
_{1}o\left(  \left\vert \left\vert h\right\vert \right\vert _{1}\right)  ,
\]
where the function $o\left(  \left\vert \left\vert h\right\vert \right\vert
_{1}\right)  $ is small uniformly in $t\leq T$ and $\mu^{1},$ provided $T$ is
small enough.
\end{proposition}

\begin{proof}
To write the equation for the Frechet differential $h\left(  t\right)
=\left[  \mathcal{D}\mu\left(  t\right)  \right]  \left(  h\right)  $ of the
map $\mu\left(  t\right)  $ at the point $\mu=\mu\left(  0\right)  $ in the
direction $h$ we have to compare the evolving measures $\mu^{1}\left(
t\right)  $ and $\mu^{2}\left(  t\right)  ,$ which are solutions of the
equation $\left(  \ref{001}-\ref{02}\right)  $ with initial conditions $\mu$
and $\mu+h,$ and keep the terms linear in $h.$ In what follows we use the
notation $\sigma_{tr}^{\mu},$ where the superscript refers to the state in
which the rate\textbf{ }$\sigma_{tr}$ is computed, see $\left(  \ref{0012}%
\right)  .$

We have%
\begin{align}
&  \frac{d}{dt}h_{v}\left(  q_{v},t\right)  =\label{0128}\\
&  =-\frac{d}{dr_{i^{\ast}\left(  q_{v}\right)  }\left(  q_{v}\right)  }%
h_{v}\left(  q_{v},t\right)  +\nonumber
\end{align}
(derivative along the direction $r\left(  q_{v}\right)  $)%
\begin{align*}
&  +\delta\left(  0,\tau\left(  e\left(  q_{v}\right)  \right)  \right)
h_{v}\left(  q_{v}\ominus e\left(  q_{v}\right)  \right)  \left[  \sigma
_{tr}^{\mu}\left(  q_{v}\ominus e\left(  q_{v}\right)  ,q_{v}\right)
+\sigma_{e}\left(  q_{v}\ominus e\left(  q_{v}\right)  ,q_{v}\right)  \right]
\\
&  +\delta\left(  0,\tau\left(  e\left(  q_{v}\right)  \right)  \right)
\mu_{v}\left(  q_{v}\ominus e\left(  q_{v}\right)  \right)  \left[
\sigma_{tr}^{h}\left(  q_{v}\ominus e\left(  q_{v}\right)  ,q_{v}\right)
\right]  -
\end{align*}
$q_{v}$ is created from $q_{v}\setminus e\left(  q_{v}\right)  $ by the
arrival of $e\left(  q_{v}\right)  $ from $v^{\prime},$ $\delta\left(
0,\tau\left(  e\left(  q_{v}\right)  \right)  \right)  $ accounts for the fact
that if the last customer $e\left(  q_{v}\right)  $ was already served for
some time, than he cannot arrive from the outside, see $\left(  \ref{0012}%
\right)  $ and $\left(  \ref{42}\right)  ;$
\begin{equation}
-h_{v}\left(  q_{v},t\right)  \sum_{q_{v}^{\prime}}\left[  \sigma_{tr}^{\mu
}\left(  q_{v},q_{v}^{\prime}\right)  +\sigma_{e}\left(  q_{v},q_{v}^{\prime
}\right)  \right]  -\mu_{v}\left(  q_{v},t\right)  \sum_{q_{v}^{\prime}%
}\left[  \sigma_{tr}^{h}\left(  q_{v},q_{v}^{\prime}\right)  \right]
+\nonumber
\end{equation}
(the queue $q_{v}$ is changing due to customers arriving from the outside and
from other servers)%
\[
+\left[  \int_{q_{v}^{\prime}:q_{v}^{\prime}\ominus C\left(  q_{v}^{\prime
}\right)  =q_{v}}dh_{v}\left(  q_{v}^{\prime}\right)  \sigma_{f}\left(
q_{v}^{\prime},q_{v}^{\prime}\ominus C\left(  q_{v}^{\prime}\right)  \right)
\right]  -h_{v}\left(  q_{v}\right)  \sigma_{f}\left(  q_{v},q_{v}\ominus
C\left(  q_{v}\right)  \right)  +
\]
(here the first term describes the creation of the queue $q_{v}$ after a
customer was served in a queue $q_{v}^{\prime}$ (longer by one customer), such
that $q_{v}^{\prime}\ominus C\left(  q_{v}^{\prime}\right)  =q_{v},$ while the
second term describes the completion of service of a customer in $q_{v}$)
\begin{equation}
+\sum_{v^{\prime}\text{n.n.}v}\beta_{vv^{\prime}}\left[  h_{v}\left(
q_{v}\right)  -h_{v}\left(  q_{v}\right)  \right]  \label{0102}%
\end{equation}
(the $\beta$-s are the rates of exchange of the servers).

The existence of the solution to the (linear) equation $\left(  \ref{0128}%
-\ref{0102}\right)  $ follows by the Peano theorem, while the uniqueness of
the solution is implied by the estimate
\[
\left\Vert h\left(  t\right)  \right\Vert \leq\left\Vert h\left(  0\right)
\right\Vert e^{Ct},
\]
which follows from the Gronwall estimate.

Finally we want to estimate the remainder,%
\[
\zeta\left(  t\right)  =\left[  \mu+h\right]  \left(  t\right)  -\mu\left(
t\right)  -\left[  \mathcal{D}\mu\left(  t\right)  \right]  \left(  h\right)
.
\]
Here $\mu+h\equiv\mu\left(  0\right)  +h\left(  0\right)  \equiv\left[
\mu+h\right]  \left(  0\right)  $ is the small perturbation of $\mu,$ $\left[
\mu+h\right]  \left(  t\right)  $ is its evolution, and $\left[
\mathcal{D}\mu\left(  t\right)  \right]  \left(  h\right)  $ is the
application of the Frechet differential of the map $\nu\left(  0\right)
\rightsquigarrow\nu\left(  t\right)  ,$ computed at the point $\mu$ and
applied to the increment $h.$ Note that $\zeta\left(  0\right)  =0$ and it
satisfies the equation%
\begin{align*}
&  \frac{d}{dt}\zeta_{v}\left(  q_{v},t\right) \\
&  =-\frac{d}{dr_{i^{\ast}\left(  q_{v}\right)  }\left(  q_{v}\right)  }%
\zeta_{v}\left(  q_{v},t\right)  +\\
&  +\delta\left(  0,\tau\left(  e\left(  q_{v}\right)  \right)  \right)
\zeta_{v}\left(  q_{v}\ominus e\left(  q_{v}\right)  \right)  \sigma
_{e}\left(  q_{v}\ominus e\left(  q_{v}\right)  ,q_{v}\right)  -\\
&  +\delta\left(  0,\tau\left(  e\left(  q_{v}\right)  \right)  \right)
\zeta_{v}\left(  q_{v}\ominus e\left(  q_{v}\right)  \right)  \sigma
_{tr}^{\left[  \mu+h\right]  }\left(  q_{v}\ominus e\left(  q_{v}\right)
,q_{v}\right) \\
&  +\delta\left(  0,\tau\left(  e\left(  q_{v}\right)  \right)  \right)
\left[  \mu+h\right]  _{v}\left(  q_{v}\ominus e\left(  q_{v}\right)  \right)
\sigma_{tr}^{\zeta}\left(  q_{v}\ominus e\left(  q_{v}\right)  ,q_{v}\right)
\\
&  +\delta\left(  0,\tau\left(  e\left(  q_{v}\right)  \right)  \right)
h_{v}\left(  q_{v}\ominus e\left(  q_{v}\right)  \right)  \sigma_{tr}%
^{h}\left(  q_{v}\ominus e\left(  q_{v}\right)  ,q_{v}\right)  -\\
&  -\zeta_{v}\left(  q_{v},t\right)  \sum_{q_{v}^{\prime}}\left[  \sigma
_{tr}^{\left[  \mu+h\right]  }\left(  q_{v},q_{v}^{\prime}\right)  +\sigma
_{e}\left(  q_{v},q_{v}^{\prime}\right)  \right]  -\\
&  -h_{v}\left(  q_{v},t\right)  \sum_{q_{v}^{\prime}}\sigma_{tr}^{h}\left(
q_{v},q_{v}^{\prime}\right)  -\left[  \mu+h\right]  _{v}\left(  q_{v}%
,t\right)  \sum_{q_{v}^{\prime}}\sigma_{tr}^{\zeta}\left(  q_{v},q_{v}%
^{\prime}\right)  +\\
&  +\left[  \int_{q_{v}^{\prime}:q_{v}^{\prime}\ominus C\left(  q_{v}^{\prime
}\right)  =q_{v}}d\zeta_{v}\left(  q_{v}^{\prime}\right)  \sigma_{f}\left(
q_{v}^{\prime},q_{v}^{\prime}\ominus C\left(  q_{v}^{\prime}\right)  \right)
\right]  -\zeta_{v}\left(  q_{v}\right)  \sigma_{f}\left(  q_{v},q_{v}\ominus
C\left(  q_{v}\right)  \right)  +\\
&  +\sum_{v^{\prime}\text{n.n.}v}\beta_{vv^{\prime}}\left[  \zeta_{v^{\prime}%
}\left(  q_{v}\right)  -\zeta_{v}\left(  q_{v}\right)  \right]  .
\end{align*}
Note that the initial condition for the last equation is $\left[
\zeta\right]  _{v}\left(  q,t\right)  =0.$ The terms in the r.h.s. which do
not contain $\zeta$ are%
\begin{align*}
&  \delta\left(  0,\tau\left(  e\left(  q_{v}\right)  \right)  \right)
h_{v}\left(  q_{v}\ominus e\left(  q_{v}\right)  \right)  \sigma_{tr}%
^{h}\left(  q_{v}\ominus e\left(  q_{v}\right)  ,q_{v}\right) \\
&  -h_{v}\left(  q_{v},t\right)  \sum_{q_{v}^{\prime}}\sigma_{tr}^{h}\left(
q_{v},q_{v}^{\prime}\right)  ,
\end{align*}
which are of the order of $\left\vert \left\vert h\right\vert \right\vert
^{2}.$ Therefore by Gronwall inequality the same bound holds uniformly for the
function $\left[  \zeta\right]  _{v}\left(  q,t\right)  ,$ provided $t\leq T$
with $T$ small enough.
\end{proof}

\begin{proposition}
The set of uniformly differentiable functions is a core of the generator of
our semigroup.
\end{proposition}

\begin{proof}
Follows from the previous Proposition, by the chain rule, and the
Stone-Weierstrass Theorem.
\end{proof}

This implies that the second condition $\left(  \ref{013}\right)  $ of the
Theorem \ref{TK} holds as well, so in our case it is indeed applicable.

\section{Conclusion}

In this paper we have established the convergence of the mean-field version of
the spatially extended network with jumping servers to a Non-Linear Markov
Process. The configuration of the $N$-component mean-field network is
described by the (atomic) measure $\mu_{N}\left(  t\right)  ,$ which randomly
evolves in time. We have shown that in the limit $N\rightarrow\infty$ the
measures $\mu_{N}\left(  t\right)  \rightarrow\mu\left(  t\right)  ,$ where
the evolution $\mu\left(  t\right)  $ is already non-random. In a sense, this
result can be viewed as a functional law of large numbers.

Our results can easily be generalized to the situation when instead of the
underlying (infinite) graph $G$ we take a sequence of finite graphs $H_{n},$
such that $H_{n}\rightarrow G,$ consider the $N$-fold mean-field type networks
$H_{n,N},$ and take the limit as $n,N\rightarrow\infty.$

\end{document}